\documentclass[a4paper,11pt]{amsart}
\usepackage{amssymb}
\usepackage{graphicx} 
\usepackage{color}
\usepackage{mathrsfs}
\usepackage{txfonts}
\usepackage{soul}

\newcounter{contai}

\newtheorem{theorem}{Theorem}
\newtheorem{corollary}[theorem]{Corollary}
\newtheorem{definition}{Definition}
\newtheorem{lemma}[theorem]{Lemma}
\newtheorem{proposition}[theorem]{Proposition} 
\newtheorem{remark}{Remark} 
\newcommand{\NN}{{\rm\bf N}}
\newcommand{\SL}{\text{SL}}

\newcommand{\RR}{{\rm\bf R}}

\newtheorem{maintheorem}{Theorem}

\begin{document}

\title[Dynamics of conservative Bykov cycles]{Dynamics of conservative Bykov cycles:\\ tangencies, generalized cocoon bifurcations\\ and elliptic solutions}


\author[M\'ario Bessa]{M\'ario Bessa}
\address[M. Bessa]{Departamento de Matem\'atica, Univ. da Beira Interior,\\
Rua Marqu\^es d'\'Avila e Bolama\\ 6201-001 Covilh\~a, Portugal\\}

\author[Alexandre Rodrigues]{Alexandre Rodrigues}
\address[A.A.P. Rodrigues]{Centro de Matem\'atica da Universidade do Porto\\ 
and Faculdade de Ci\^encias da Universidade do Porto\\
Rua do Campo Alegre 687, 4169--007 Porto, Portugal}

\email[M\'ario Bessa]{bessa@ubi.pt}
\email[Alexandre Rodrigues]{alexandre.rodrigues@fc.up.pt}

\thanks{CMUP is supported by the European Regional Development Fund through the programme COMPETE and by the Portuguese Government through the Funda\c{c}\~ao para a Ci\^encia e a Tecnologia (FCT) under the project PEst-C/MAT/UI0144/2011.  AR was supported by the grant SFRH/BPD/84709/2012 of FCT. Part of this work has been written during AR stay in Nizhny Novgorod University supported by the grant RNF 14-41-00044 }

\begin{abstract}
This paper presents a mechanism for the coexistence of hyperbolic and non-hy\-per\-bolic dynamics arising in a neighbourhood of a conservative Bykov cycle where trajectories turn in opposite directions near the two saddle-foci. We show that {within the class of divergence-free vector fields that preserve the cycle,} tangencies of the invariant manifolds of two hyperbolic saddle-foci densely occur. 
The global dynamics is persistently dominated by heteroclinic tangencies and by the existence of infinitely many elliptic points coexisting with suspended hyperbolic horseshoes. A generalized version of the Cocoon bifurcations for conservative systems is obtained.
\end{abstract}

\maketitle

\section{Introduction}\label{intro}

Homo and heteroclinic bifurcations constitute the core of our understanding of complicated recurrent behaviour in dynamical systems. The history goes back to Poincar\'e in the late $19^{th}$ century, with major subsequent contributions by the schools of Andronov, Shilnikov, Smale and Palis. These schools have been founded on a combination of analytical and geometrical tools which developed a quite good understanding of the qualitative behaviour of those dynamics.

Differential equations modelling physical experiments frequently have parameters which appear
in the differential equations, and it is known that qualitative changes
may occur in the solution structure of these systems as the parameters vary. The dynamical behaviour of systems can be strongly
influenced by special geometric or analytic invariant structures appearing in the equations (\emph{e.g.} divergence-free, preserving or reversing symmetries).
One may ask which dynamical behaviour we would expect to see in the presence of a given invariant 
structure. Generically, this is a hard question to answer, but some questions could be partially answered, by considering local and global bifurcations of low codimension. 

\subsection{The problem}
{A Bykov cycle on a three-dimensional manifold is a heteroclinic cycle between two hyperbolic saddle-foci of different Morse index, where one of the connections is transverse and the other is structurally unstable. There are two different possibilities for the geometry of the flow around the one-dimensional connection depending on the direction solutions turn around it. All literature about the Michelson system \cite{DIK1, DIK, DIKS, KE, KWZ, LR, LTW, Michelson, Webster} considered that in the neighbourhoods of the two saddle-foci, trajectories wind in the same direction. An immediate question arises:}
\begin{itemize}
\item[\textbf{(Q1)}]  {What happens if trajectories wind with opposite directions near each node? }
\end{itemize}

{The first author and Duarte proved in~\cite{BesDu07} that the set of $C^1$-divergence-free vector fields defined in a compact three-dimensional Riemannian manifold without boundary, has a $C^1$-residual set such that any vector field inside it is Anosov or else, the flow associated to it has dense elliptic solutions in the phase space. Furthermore, in \cite{BessaRocha}, also in this context, the authors proved that if the vector field is not Anosov, then it can be $C^1$-approximated by another divergence-free vector field exhibiting homoclinic tangencies. However, a conservative vector field whose flow has a persistent Bykov cycle may lie outside these residual/dense subsets and the developed theory cannot be applied for this degenerated class of systems. Since Bykov cycles cannot be Anosov (due to the presence of equilibria), the dichotomies in \cite{BessaRocha, BesDu07} suggest the following problem:}
\begin{itemize}
\item[\textbf{(Q2)}]  {Could we perform a $C^1$-perturbation within the set of vector fields whose flow has a Bykov cycle in such a way that the elliptic periodic solutions are dense and/or tangencies occur? }
\end{itemize}

{In this paper, we partially answer the questions \textbf{(Q1)}--\textbf{(Q2)}.} We present a mechanism for the coexistence of hyperbolic and non-hyperbolic dynamics arising in a neighbourhood of a conservative Bykov cycle where trajectories turn in opposite directions near the two saddle-foci - see Figure \ref{conservativo2}. We show that in a $C^1$-open class of divergence-free vector fields, tangencies of the invariant manifolds of two hyperbolic saddle-foci densely occur. We prove that the global dynamics is persistently dominated by heteroclinic tangencies and by the existence of infinitely many elliptic points coexisting with hyperbolic dynamics arising from $C^1$--transversality.
Tangencies are strongly connected with the Cocoon bifurcations. We describe an extended version of these bifurcations which can be detected numerically  with the \emph{Time Delay function} introduced by Lau \cite{Lau}. See also \cite{DIK, DIKS}. We also find the relations between these four dynamical phenomena: chirality, tangencies, generalized Cocoon bifurcations and elliptic solutions. In the conservative setting, Cocoon bifurcations are connected with unfoldings of a Hopf-zero singularity, which has been shown to occur in climatological models with seasonal forcing.

\begin{figure}\begin{center}
\includegraphics[height=3cm]{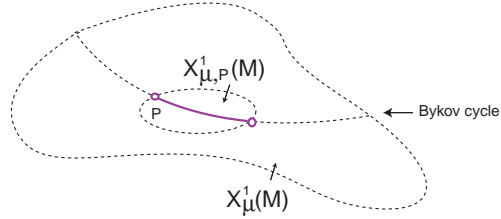}
\end{center}
\caption{ \small Scheme of the space of vector fields under consideration.  $\mathfrak{X}_{\mu}^1(M)$ represents the set of $C^1$-divergence free vector fields on $M$. On the dashed line, we find the vector fields whose flow has a given Bykov cycle and its continuation. The set $\mathfrak{X}_{\mu, P}^1(M)$ is the open set (with the induced topology) in which Theorem \ref{Main1} is valid. }
\label{scheme1}
\end{figure}

\subsection{Elliptic periodic solutions revisited}
We recall some important results about heteroclinic tangencies for general and conservative systems. The systematic study of bifurcations of tangencies was started by Gavrilov and Shilnikov \cite{GS1, GS2} in the seventies, for the case of two-dimensional dissipative diffeomorphisms. The authors established that diffeomorphisms with homoclinic tangencies might separate systems with regular and chaotic dynamics, \emph{i.e.} they belong to the boundary of Morse-Smale systems and the transition through this boundary corresponds to an $\Omega$--explosion. 

The theorem on the cascade of periodic sinks/sources near a tangency plays an important role -- see Newhouse  \cite{Newhouse2}:  in the dissipative case, for any one parameter family that unfolds generically a quadratic homoclinic tangency, there exists a sequence of intervals of values of the parameter such that the corresponding diffeomorphism has sinks/sources. One of the fundamental results in homoclinic bifurcations was established in Newhouse \cite{Newhouse2, Newhouse1}: the existence of regions of the space of two-dimensional diffeomorphisms where tangencies are dense. These regions are called \emph{Newhouse regions}; the result has been generalized for the multidimensional case by Gonchenko \cite{Gonc83}. Dynamics of systems within these regions is complex and rich and, as claimed in \cite{GST99}, it is impossible to give the complete description of bifurcations of such systems. 

The majority of these results was obtained for the broader case of general systems without any restriction on the preservation of some invariant structure. They cannot be directly applied to volume-preserving  and/or reversible systems because they require special considerations. Nevertheless, the main geometric and analytical arguments can be also used for systems with additional structures, having special care. In this direction, important results on the birth of elliptic periodic points in area-preserving maps under bifurcations of tangencies were obtained in \cite{Duarte, GS, Newhouse3}. See also the paper by Lamb and Stenkin \cite{LS2004} about \emph{mixed dynamics.}

Bifurcations of single-round periodic solutions were studied in \cite{MGonchenko, GS} for two-dimensional symplectic maps close to a map having a quadratic homoclinic tangency. The corresponding first return maps have been derived, bifurcations of their fixed points have been studied and bifurcation diagrams for one-parameter general unfoldings have been constructed. Using renormalization results, the existence of one-parameter cascades of elliptic single-round periodic solutions has been proved. 
\medbreak

\section{Preliminaries}
\subsection{Divergence-free vector fields}
Let $M$ be a three-dimensional closed and connected $C^\infty$ Riemaniann manifold without boundary, endowed with a volume-form and let $\mu$ denote the Lebesgue measure associated to it.  
 Let $F: M \rightarrow TM$ be a $C^r$  vector field, with $r\geq 1$. By Picard's theorem on the existence and uniqueness of solutions of differential equations, $F$ integrates into a complete $C^r$ flow $\varphi(t, x)\colon \mathbb{R} \times M\rightarrow M$ in a sense that $$\frac{d}{dt}\varphi(t, x)|_{t=s}=F(\varphi(s, x)).$$
We say that $F=(F_1, F_2, F_3)$ is \emph{divergence-free} if its divergence is equal to zero \emph{i.e.} if $$\nabla\cdot F:=\frac{\partial F_1}{\partial x_1}+\frac{\partial F_2}{\partial x_2}+\frac{\partial F_3}{\partial x_3}=0$$ computed in local coordinates $x=(x_1,x_2,x_3)\in M$, see \cite{Mo}. Equivalently, $F$ is divergence-free if, accordingly to the Liouville formula,  
the volume-measure $\mu$ is invariant for the associated flow. In this case we say that the flow is \emph{conservative}, \emph{volume-preserving} or \emph{incompressible} and is such that $\det D\varphi(t,x)=1$, for all $ x \in M$ and for all $ t \in \mathbb{R}$.
\medbreak

We denote by $\mathfrak{X}_\mu^r(M)$, $r\geq 1$, the space of $C^r$ divergence-free vector fields on $M$ and we endow this set with the usual $C^r$ Whitney topology. Let also denote by $\mathfrak{X}^r(M)\supset \mathfrak{X}_{\mu}^r(M)$ the space of $C^r$ general vector fields on $M$ without no divergence-free constraint. 
We say that an equilibrium $\sigma\in M$ of $F$  (\emph{i.e.} $F(\sigma)=\vec0$), is \emph{hyperbolic} if $DF(\sigma)$ has neither pure complex  nor zero as an eigenvalue. A hyperbolic saddle-equilibrium  with a pair of conjugated non-real eigenvalues is called a \emph{saddle-focus}. For regular points, (\emph{i.e.} non-equilibrium points) we define hyperbolicity with respect to the Poincar\'e map in a standard way -- see e.g. ~\cite{Katok}. Given a hyperbolic equilibrium point $p\in M$ with respect to a vector field $F$, we denote by $W^{u/s}(F,p)$ its unstable/stable global manifold. Local unstable/stable manifolds are denoted by $W^{u/s}_{loc}(F,p)$. When there is no ambiguity we omit the letter $F$ in these notations.

\begin{figure}\begin{center}
\includegraphics[height=7cm]{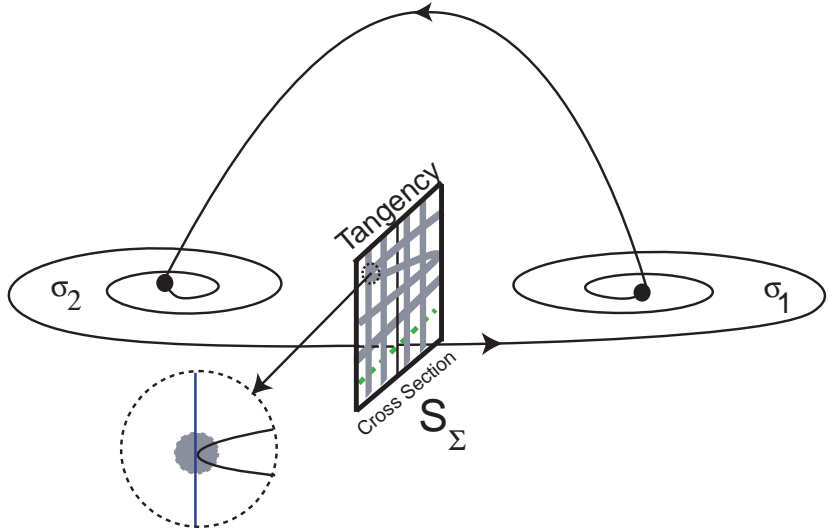}
\end{center}
\caption{ \small Bykov cycle scheme satisfying \textbf{(H1)}--\textbf{(H5)}: in a {$C^1$}--open class of divergence-free vector fields defined on a three-dimensional compact manifold, tangencies of the invariant manifolds of two hyperbolic saddle-foci occur densely.}
\label{conservativo2}
\end{figure}

\subsection{Bykov cycle}

In the present paper, all equilibria are hyperbolic saddle-foci.  The dimension of the unstable manifold of a saddle-focus will be called the \emph{Morse index} of the saddle.  Given two equilibria $\sigma_1$ and $\sigma_2$, a \emph{heteroclinic trajectory} from $\sigma_1$ to $\sigma_2$, denoted 
$[\sigma_1\rightarrow \sigma_2]$, is a solution of $\dot{x}=F(x)$ contained in $W^{u}(\sigma_1)\cap W^{s}(\sigma_2)$. There may be more than one solution from  $\sigma_1$ to $\sigma_2$. A heteroclinic cycle is a finite collection of equilibria together with a set of heteroclinic trajectories connecting the equilibria in a cyclic way. More details in Field \cite{Field}.

In a three-dimensional smooth manifold, a \emph{Bykov cycle} is a heteroclinic cycle associated to two hyperbolic saddle-foci with different Morse indices, in which the one-dimensional manifolds coincide and the two-dimensional invariant manifolds have a transverse intersection \cite{Bykov, LR}. {The terminology \emph{Bykov cycle} is because it was Bykov \cite{Bykov93, Bykov99, Bykov} who studied dynamical properties from the existence of similar cycles in a more general context than conservative ones. This type of cycle, or rather its bifurcation point, is also called by \emph{T-point} by Glendinning and Sparrow \cite{GS}.}

\section{Overview of the main results}
In this section, we state the main theorems and we relate them with other results in the literature.
\subsection{Description}
\label{description}
Our object of study is the dynamics around an autonomous divergence-free smooth vector field whose flow has a Bykov cycle, for which we give an accurate description here. Specifically, we study a {$C^1$}--vector field:
\begin{equation}
\label{general}
F:M\rightarrow TM
\end{equation}
 whose flow has the following properties (see Figure \ref{conservativo2}):
\bigbreak   
\begin{enumerate}
\item[\textbf{\textbf{(H1)}}]\label{H1} There are two hyperbolic saddle-foci equilibria $\sigma_1$ and $\sigma_2$. Assuming that, for $
j\in \{1,2\}$, the Morse index of $\sigma_j$ is $j$, the eigenvalues of $DF(\sigma)$ are:

\begin{enumerate}
\item[\textbf{(a)}] $- C_1 \pm \alpha_1 i$ and $E_1$,  where $2C_1=E_1>0$ and $\alpha_1>0$, for $\sigma= \sigma_1$;
\item[\textbf{(b)}] $ E_2 \pm \alpha_2 i$ and $-C_2$, where $C_2=2E_2>0$ and $\alpha_2>0$, for $\sigma= \sigma_2$.
\end{enumerate}
\bigbreak
\item[\textbf{\textbf{(H2)}}] \label{H2} There is a heteroclinic cycle $\Sigma$ consisting of $\sigma_1$, $\sigma_2$ and two one-dimensional heteroclinic connections: $[\sigma_1 \rightarrow \sigma_2]$ and $[\sigma_2 \rightarrow \sigma_1]$. The solution $[\sigma_1 \rightarrow \sigma_2]$ is called the \emph{fragile connection}.
\bigbreak
\item[\textbf{\textbf{(H3)}}] \label{H3} At the heteroclinic connection $[\sigma_2 \rightarrow \sigma_1]$, the two-dimensional manifolds $W^u(\sigma_2)$ and $W^s(\sigma_1)$ meet transversely.

\bigbreak
\end{enumerate}

Property \textbf{(H3)} 
is generic. Indeed, the transverse intersection of $W^u(\sigma_2)$ and $W^s(\sigma_1)$ of \textbf{(H3)} persists under $C^1$--perturbations, whereas the fragile connection $[\sigma_1 \rightarrow \sigma_2]$ does not. Although Bykov cycles appear naturally in systems with some symmetry as in   Knobloch \emph{et al} \cite{KLW}, Michelson \cite{Michelson}, Rodrigues and Labouriau \cite{RL2014}, unless it is explicitly said, in this paper we state general results without any kind of restriction about the symmetry. 
There are two  possibilities for the geometry of the solutions around the cycle, depending on the direction they turn around the one-dimensional heteroclinic connection $[\sigma_1\rightarrow \sigma_2]$, in the neighbourhoods of the saddle-foci. Under the context of hypotheses \textbf{(H1)}--\textbf{(H3)}, we introduce the following definition adapted from \cite{LR3}. 
\medbreak

Let $V_1$
and $V_2$ be small disjoint neighbourhoods of two saddle-foci $\sigma_1$ and $\sigma_2$ with disjoint boundaries $\partial V_1$ and $\partial V_2$, respectively. 
Typical trajectories starting at $y\in \partial V_1$ near $W^s(\sigma_1)$ go into the interior of $V_1$ in positive time, then follow the connection $[\sigma_1\rightarrow\sigma_2]$, go inside $V_2$, and then come out at $\partial V_2$.  Let $\varphi$ be a piece of the described trajectory from $\partial V_1$ to $\partial V_2$.
Now, join its starting point to its end point by a line segment, forming a closed curve. This curve is called the \emph{loop of} $y$.

\begin{definition}\label{Chi}
The two saddle-foci $\sigma_1$ and $\sigma_2$ in $\Sigma$ have the \emph{same chirality} if the loop of every trajectory in $\partial V_1$ is linked to $\Sigma$ in the sense that, for every $y$  close to $W^s_{loc}(\sigma_1)$, the loop of $y$ and $\Sigma$ cannot be disconnected by an isotopy. Otherwise, we say that $\sigma_1$ and $\sigma_2$ have \emph{different chirality}.
\end{definition}

In contrast with the findings of \cite{KLW, LR, LTW, Rodrigues2}, if the two nodes have different chirality, then the rotations may cancel out. This is a key idea for the proof of the main results, which will be formalized via hypothesis \textbf{(H4)}.
\bigbreak
\begin{enumerate}
\item[\textbf{\textbf{(H4)}}] \label{H4} The saddle-foci $\sigma_1$ and $\sigma_2$ have different chirality.
\end{enumerate}

\bigbreak
A good overview of the previous hypothesis has been considered in \cite{LR3}. It corresponds to the \emph{non-concatenation property} stated in \cite{Rodrigues2}. Property \textbf{(H4)} means that there are two open neighbourhoods of the equilibria such that for any
trajectory going from the first to the second, the direction of its turning around the heteroclinic connection $[\sigma_1 \rightarrow \sigma_2]$ is different. A direct corollary is the following:
\begin{corollary}
The condition \textbf{(H4)} is persistent under \emph{isotopies}: if it holds for $F$, then it holds for continuous one-parameter families containing it, as long as the fragile connection remains.
\end{corollary}
In order to prove our main results, we need some additional assumptions on  the set of parameters $P=\left(\alpha_1, C_1, E_1, \alpha_2, C_2, E_2\right),$
that determine the linear part of the vector field $F$ at the equilibria. 
{For any $a\geq 1$, let $\mathscr{B}$ be the subset of parameters  given by:
\begin{equation}
\label{final_formulae}
\mathscr{B}=\left\{ P:\ 
\left(a^2-\frac{1}{a^2}\right)\frac{2\alpha_1}{C_1 - \sqrt{\alpha_1^2+ 4C_1^2}} 
< \frac{E_2}{\alpha_2}- \frac{a^2 C_1}{\alpha_1} <  
\left(a^2-\frac{1}{a^2}\right)\frac{2\alpha_1}{C_1 + \sqrt{\alpha_1^2+4 C_1^2}}
\right\} .
\end{equation}}
\bigbreak
{Our last hypothesis if the following: 
\bigbreak
\begin{enumerate}
\item[\textbf{(H5)}] \label{H5}
For all $a>1$, $\mathscr{B}$ has non-empty topological interior in $\mathbb{R}^6$ (\emph{i.e.} $int(\mathscr{B}) \neq \emptyset$).
\end{enumerate}
\bigbreak
 Let $\mathscr{D}$ be the dense subset  of $\mathscr{B}$ given by:
\begin{equation}\label{denseSet}
\mathscr{D}=\left\{ P\in int(\mathscr{B}):\ 
\gamma=\frac{\alpha_2}{\alpha_1}\frac{C_1}{E_2} \notin \mathbb{Q}
\right\} .
\end{equation}}

The constant $a$ is related with the construction of the global map around the cycle and will be clarified later.
Observe that condition (\ref{final_formulae}) in the definition of  $\mathscr{B}$ is satisfied by an open set of parameters 
$\alpha_1$, $C_1$, $\alpha_2$, $E_2$ and does not involve the quantities $E_1$ and $C_2$.  The set $\mathscr{D}$ has full Lebesgue measure within the set of parameters $\mathscr{B}$. Hereafter, let $\mathfrak{X}_{\mu,P}^1(M)\subset \mathfrak{X}_{\mu}^1(M)$ denote the set of  vector fields on $M$ satisfying \textbf{\textbf{(H1)}}--\textbf{(H5)}, which is a Baire space \cite{Webster}. {As depicted in Figure \ref{scheme1}, the set $\mathfrak{X}_{\mu,P}^1(M)$ may be seen as an open set in the space $\mathfrak{X}_{\mu}^1(M)$, endowed with the induced topology. }

Let $c_K= 15\sqrt{\frac{22}{19^3}}$. The vector field $F_{c_K}$ in Dumortier \emph{et al} \cite{DIK1} associated to the Michelson system satisfies {\textbf{(H1)}}--{\textbf{(H3)}} but condition \textbf{(H4)} does not hold. 
{$R$-reversibility where $\dim Fix(R)=1$ is a natural obstacle to \textbf{(H4)}.}
In the next section we show a way to construct theoretically a divergence-free vector field displaying a Bykov cycle and satisfying \textbf{\textbf{(H1)}}--\textbf{\textbf{(H5)}}. This is what we call a \emph{conservative Bykov cycle}. 

\subsection{The construction}

\begin{proposition}\label{BykovExists}
Any closed, connected, Riemannian three-dimensional manifold $M$ supports a Bykov he\-te\-roclinic cycle satisfying properties {\textbf{(H1)}}--\textbf{(H5)} associated to some divergence-free vector field $\mathfrak{X}^r_{\mu}(M)$, $r\geq 1$.
\end{proposition}

\begin{proof}
Let us consider a Riemannian closed, connected and three-dimensional manifold $M$. We let $\sigma_1, \sigma_2\in \mathbb{R}^3$ be two hyperbolic saddle-foci associated to linear divergence-free vector fields $L_1$ and $L_2$ and complying the hypotheses \textbf{\textbf{(H1)}} and \textbf{(H5)}, and moreover in such a way that \textbf{(H4)} also holds. Now, we would like to plunge in $M$ the open sets containing $\sigma_1$ and $\sigma_2$. 

By the Poincar\'e-Hopf theorem ~\cite{GP}, we know that the sum of the indices of the equilibria of a vector field equals the Euler characteristic of $M$. So, we carefully choose a finite number $(k-2)$ of traceless equilibria, \emph{i.e.} equilibria  associated to divergence-free linear vector fields $\{L_j\}_{j=3}^k$ and fulfilling the Poincar\'e-Hopf Theorem.

Using Moser's charts \cite{Mo} we send (locally) the vector fields $\{L_j\}_{j=1}^k$ and then, using the Pasting lemma \cite{ArMa} we extend to a divergence-free vector field in the whole manifold. Let $F\in\mathfrak{X}^{{r}}_{\mu}(M)$ be the obtained vector field and we still denote by $\sigma_1, \sigma_2\in M$ be two hyperbolic saddle-foci with the stable manifold of $\sigma_j$ having dimension $j$, $j=1,2$.

Using \cite{Bes} we consider $F_1\in\mathfrak{X}^{{r}}_{\mu}(M)$ $C^1$-close to $F$ with a dense trajectory with initial condition $x\in M$, $\varphi_1(t,x)$. Since $\varphi_1(t,x)$ is dense, it passes arbitrarily close to $W^{u}(F_1,\sigma_2)$ and to $W^{s}(F_1,\sigma_1)$. By the $C^1$-Connecting Lemma \cite{WX} we consider $F_2\in\mathfrak{X}^{{r}}_{\mu}(M)$ $C^1$-close to $F_1$ such that $W^{u}(F_2,\sigma_2)$ intersects $W^{s}(F_2,\sigma_1)$. Since these manifolds are two-dimensional and $M$ is three-dimensional we can assume that the intersection is transversal. 

Using again \cite{Bes}, we consider $F_3\in\mathfrak{X}^{{r}}_{\mu}(M)$ $C^1$-close to $F_2$ with some dense trajectory for the flow $\varphi_3(t,.)$ and with $W^{u}(F_3,\sigma_2)$ still intersecting $W^{s}(F_3,\sigma_1)$ transversely. Since the solution is dense, it passes arbitrarily close to $W^{s}(F_3,\sigma_2)$ and to $W^{u}(F_3,\sigma_1)$. Again, by the $C^1$-Connecting Lemma, we consider $F_4\in\mathfrak{X}^{{r}}_{\mu}(M)$ $C^1$-close to $F_3$ such that $W^{s}(F_4,\sigma_2)$ intersects $W^{u}(F_4,\sigma_1)$, non-transversely of course, and with $W^{u}(F_4,\sigma_2)$ still intersecting $W^{s}(F_4,\sigma_1)$ transversely. This satisfies the hypotheses \textbf{(H2)}  and \textbf{(H3)} and therefore we obtain a Bykov heteroclinic cycle associated to the divergence-free vector field $F_4$. 
\end{proof}

{By technical reasons, in Section~\ref{SectionLocalDynamics}, our main results ask for at least $C^4$-regularity of the vector field \cite[pp. 616]{Banyaga}. Nevertheless, if we start with an element in $\mathfrak{X}^1_{\mu}(M)$ exhibiting a Bykov cycle and since the hyperbolic equilibria are $C^1$-stable, similar arguments to those used in Proposition~\ref{BykovExists} can ensure that $C^1$-arbitrarily close to the initial vector field, there exists an element in $\mathfrak{X}^4_{\mu}(M)$ satisfying {\textbf{(H1)}}--\textbf{(H5)}. This is due to the fact that $\mathfrak{X}^\infty_{\mu}(M)$  is $C^1$-dense in $\mathfrak{X}^1_{\mu}(M)$ by ~\cite[Theorem 2.2]{ArMa}. We address the reader to the statements of ~\cite[Theorem 3.1 and 3.2]{ArMa} and the toolboxes used in the $C^1$-Closing and Connecting Lemmas which do not affect the regularity of the initial vector field. }

\subsection{Main results}
By \cite{BessaRocha}, the existence of {hyperbolic} equilibria implies that the flow is not $C^1$--near an Anosov one and thus,  $C^1$--close to the flow, tangencies are expected.  Furthermore, by \cite{BesDu07}, arbitrarily $C^1$--close to a conservative flow displaying a Bykov cycle, we may find an incompressible flow with a dense set of elliptic periodic solutions. {Nonetheless,} these approximated flows might not have a Bykov cycle anymore, due to the fragile connection $ [\sigma_1 \rightarrow \sigma_2]$. {We start with some definitions that make our results more precise.}

\begin{definition}
Let $k\in\mathbb{N}\cup\{0\}$, $N_\Sigma$ be a small tubular neighbourhood of the heteroclinic cycle $\Sigma$ and $S$ be a cross section to  $[\sigma_2 \rightarrow \sigma_1]$. We say that the two-dimensional invariant manifolds $W^u(\sigma_2)$ and $W^s(\sigma_1)$ have a tangency of order $k$ inside $N_\Sigma$ if the following conditions hold:
{\begin{itemize}
\item the manifolds $W^u(\sigma_2)$ and $W^s(\sigma_1)$ meet tangencially;
\item there is a solution contained in $W^u(\sigma_2) \cap W^s(\sigma_1)$ which lies entirely in $N_\Sigma$ and has exactly $(k+1)$ intersection points with $S$. 
\end{itemize}}
A tangency of order 1 inside $N_\Sigma$ is also called by \emph{primary tangency}.
\end{definition}
Our first result is the following:

 \begin{maintheorem}
\label{Main1}
There exists a $C^1$--dense subset $\mathcal{D}$ of the $C^1-$open set $\mathfrak{X}_{\mu, P}^1(M)$ such that for any $F\in\mathcal{D}$ and any tubular neighbourhood $N_\Sigma$ of the Bykov cycle, the flow of (\ref{general}) displays a primary tangency inside $N_\Sigma$ between the manifolds $W^u(\sigma_2)$ and $W^s(\sigma_1)$.
\end{maintheorem}

{The perturbations of Theorem \ref{Main1} are performed within a degenerate} subclass of incompressible flows exhibiting Bykov cycles. The density of flows exhibiting tangencies is achieved without breaking the Bykov cycle. The proof of this result will be addressed in Section \ref{SectionT-pointHetero}, where we concentrate our attention on the geometric intersection between the invariant manifolds.
Now, we consider a generalized version of the Cocoon bifurcations described in \cite{DIK, KE, Lau}, which can be seen as a series of global and local bifurcations that create an infinite number of heteroclinic connections. The authors of  \cite{DIK, KE} called cocoon to these series of bifurcations because it forms a cocoon like structure in a section transverse to $\Sigma$.

\begin{definition}
\label{Definition_Cocoon}
Let $F \in \mathfrak{X}_{\mu,P}^1(M)$. We say that $F$ exhibits a  \emph{generalized cocooning cascade of heteroclinic tangencies} if, for every $L>0$ large, there exists a closed two-dimensional torus $\mathbb{T}^2$ such that:
\begin{enumerate}
\item for $i \in \{1,2\}$, $\sigma_i \notin \mathbb{T}^2$;
\item the vector field $F$ can be $C^1$-approximated by $F^\star \in \mathfrak{X}_{\mu,P}^1(M)$ whose flow displays a tangency between $W^u(\sigma_2)$ and $W^s(\sigma_1)$,
intersecting twice $\mathbb{T}^2$ and having length greater than $L$ within $\mathbb{T}^2$.
\end{enumerate}
\end{definition}

{Definition \ref{Definition_Cocoon} is different to the general case given in \cite{DIK} in which the authors considered generic unfoldings of $F$.} In \cite{DIK}, when following a given sequence of vector fields converging to $F$, a saddle-node bifurcation takes place, creating a pair of elliptical and hyperbolic periodic solutions, followed by an infinite cascade of period doubling bifurcations. Since there are infinitely many heteroclinic tangencies occurring near the primary tangency, in Section \ref{Cocoon_section} we prove that:

\begin{maintheorem}
\label{Main2}
Let $\mathfrak{X}_{\mu, P}^1(M)$ be the $C^1$-open set of Theorem \ref{Main1}. For any $F\in\mathfrak{X}_{\mu, P}^1(M)$, and any tubular neighbourhood $N_\Sigma$ of the Bykov cycle,  the vector field $F$ exhibits a generalized cocooning cascade of heteroclinic tangencies. 
\end{maintheorem}
The proof is based on \cite{DIK} and on the $C^1$-Connecting Lemma for volume-preserving flows ~\cite{WX}. These heteroclinic bifurcations are responsible for the folding and for the fractal structure of the two-dimensional invariant manifolds in the cross section $S$ previously described by Lay \cite{Lau}. The Cocoon bifurcations might be observed numerically with the \emph{Time Delay function}. A discussion of the mechanism explaining the scape of points is given in \cite[Section 4]{DIKS}.  

\medbreak

\begin{figure}\begin{center}
\includegraphics[height=6cm]{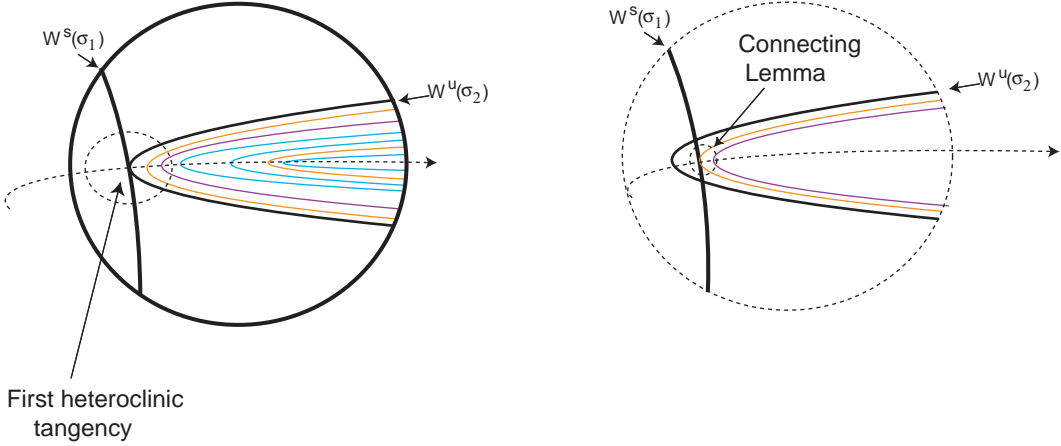}
\end{center}
\caption{\small Cocoon bifurcations: global bifurcations that create an infinite number of heteroclinic connections. Illustration of Theorem \ref{Main2}, emphasizing the region in which we apply the $C^1$-Connecting Lemma. The effect of moving the invariant two-dimensional manifolds corresponds to the motion suggested by the arrow.}
\label{Cocoon}
\end{figure}

In the context of  dissipative diffeomorphisms containing homoclinic points, Newhouse \cite{Newhouse2, Newhouse1} introduced the term \emph{wild attractor} for hyperbolic sets whose invariant manifolds have a tangency. Newhouse described what happens in a one-parameter
unfolding, when a tangency splits, and discovered nontrivial, transitive and hyperbolic sets whose stable and unstable invariant manifolds have persistent nondegenerate tangencies. A given tangency may be removed by a small perturbation, however one cannot avoid the appearance of new tangencies.  This theory can be extended to heteroclinic tangencies -- see \cite[Section 3.2]{BDV}. 

\medbreak
\begin{definition}
\label{Elliptic}
Let $R$ be a diffeomorphism defined on a compact section $M$ transverse to $\Sigma$. We say that $p\in M$ is an \emph{elliptic periodic orbit} of period $n \in \mathbb{N}$ for the diffeomorphism $R$ if the following conditions hold:
\begin{itemize}
\item $R^n(p)=p$ and $R^i(p)\not=p$, for all $i=1,...,n-1$ and
\item the eigenvalues of the map $DR^n_p \in \SL(2, \mathbb{R})$ are non-real and have modulus equal to 1.
\end{itemize}
\end{definition}
\medbreak

From now on, the map $R$ of Definition \ref{Elliptic} can be thought as the first return map to $S$ (induced by the flow). By the Generalised Stokes Theorem \cite{Marsden}, the return map of a volume-preserving vector field to any transverse section $S$ preserves an `area-measure', and so it may be seen as an area-preserving diffeomorphism -- see Lemma \ref{area-preserving}.
 Newhouse's results  near tangencies have been extended to conservative maps in \cite{Duarte}: any area-preserving map with a heteroclinic tangency can be $C^2$--approximated by an open domain in the space of area-preserving maps exhibiting persistent tangencies. Generic diffeomorphisms on such a domain exhibit infinitely many elliptic islands, which take the role of attractors in the dissipative context. 
 
 \begin{definition}
 Let $R$ be the first return map to a section transverse to the cycle $\Sigma$. A 1-periodic solution of the flow of (\ref{general}) is a periodic solution associated to a fixed point of $R$ (also called by \emph{single round periodic solution}).
  \end{definition} 
  
  We obtain the following result:
 
\begin{maintheorem}
\label{corollaryPeriodDoubling}
Let $F\in\mathfrak{X}_{\mu, P}^1(M)$.
The map $R$ defines a sequence of conservative horseshoes accumulating on $\Sigma$.  There are persistent heteroclinic tangencies of the invariant manifolds associated to periodic solutions and infinitely many 1-periodic elliptic solutions nearby.
\end{maintheorem}

The realization of $R$ as return map associated to a divergence-free vector field $C^1$-close to the original follows the same lines as \cite[Section  3]{BessaRocha}.

\medbreak
\textbf{The method.} Our analysis follows a classic procedure. We first construct a model of the dynamics near a Bykov cycle in terms of Poincar\'e maps between neighbourhoods of the saddle-foci, and then we analyze the algebraic bifurcation equations. Near each equilibrium we derive local maps by assuming that the flow is governed by appropriate linearized vector fields in these regions. The flow near the connections is approximated by linearizing about each one, which enables us to derive global maps. Appropriate composition of the local and global maps yields the desired return map. Our aim is to provide a direct and intuitive geometric picture and to uncover the scaling laws of the codimension-one bifurcations that occur near the cycle.

\section{Dynamics near each saddle-focus (Local and Global)}
\label{SectionLocalDynamics}
In this section, we establish local coordinates near the equilibria $\sigma_1$ and $\sigma_2$ and define some notation that will be used in sequel. The starting point is an application of a quite useful conservative version of Sternberg's Theorem for $C^1$--Linearization proved by Banyaga, de la Llave and Wayne \cite{Banyaga} to $C^1$--linearize the vector field around the saddle-foci and to introduce cylindrical coordinates around them. For each saddle, we obtain the expression of the local map that sends points in the border where the flow goes in, into points in the border where the flows goes out in positive time \emph{i.e.} by forward iteration. 
\medbreak
We also establish a convention for the transition maps from one neighbourhood to the other -- see also \cite[Lemma 3.4]{Bessa_ETDS} where it is proved the conservative flowbox theorem.

\subsection{Conservative linearization and new coordinates near the saddles}
\label{linearization}
Since the open and dense non-ressonance conditions of \cite[Theorem 1.2]{Banyaga} are satisfied, then around the saddle-foci the vector field $F$ is $C^1$--conjugate to its linear part. In cylindrical coordinates $(\rho ,\theta ,z)$ the linearizations at $\sigma_1 $ and $\sigma_2 $ take the form, respectively:
\begin{equation}\label{local1}
\dot{\rho}=-C_{1}\rho \qquad\wedge\qquad \dot{\theta}=\alpha_{1} \qquad \wedge \qquad \dot{z}=E_{1 }z,
\end{equation}
and
\begin{equation}\label{local2}
\dot{\rho}=E_{2 }\rho \qquad\wedge\qquad \dot{\theta}= - \alpha_{2 } \qquad \wedge \qquad \dot{z}=-C_{2 }z.
\end{equation}

Rescaling coordinates, we may consider 
cylindrical neighbourhoods of $\sigma_1 $ and $\sigma_2 $ in ${M}$ of radius $1$ and height $2$ that we denote by $V_1$ and $V_2$, respectively, as illustrated in Figure \ref{conservativo3}. Their borders have three components: 
\begin{itemize}
\item[\textbf{(i)}]
The \emph{wall} parametrized by 
$x\in \mathbb{R}\pmod{1}$ and $|y|\leq 1$ with the usual cover: 
$$(x,y)\mapsto (1 ,x,y)=(\rho ,\theta ,z).$$ Here $y$ represents the  height of the cylinder and $x$ is the angular coordinate, measured from the point $x=0$ in the heteroclinic connection $[\sigma_2\to\sigma_1]$. 
\bigbreak
\item[\textbf{(ii)}]
Two disks, the top and the bottom of the cylinder. 
We assume the fragile heteroclinic connection $[\sigma_1\to \sigma_2]$ goes from the top of one cylinder to the top of the other, and we take a polar covering of the top disk:
$$(r,\varphi )\mapsto (r,\varphi ,1 )=(\rho ,\theta ,z)$$
where 
$0\leq r\leq 1$ and $\varphi \in \mathbb{R}\pmod{2\pi}$.
\end{itemize}
\medbreak
On these cross sections, we will define the Poincar\'e maps to study the dynamics near the cycle.

\medbreak
Consider the cylinder wall  of $V_1$, that meets the stable manifold of $\sigma_1$, denoted by $W^{s}(\sigma_1)$, on the circle parametrized by $y=0$. The top part  $y\ge 0$ of the wall near $\sigma_1$ will be denoted by $In(\sigma_1)$. By construction, trajectories starting at interior points of $In(\sigma_1)\backslash W^s(\sigma_1)$ go into the cylinder in positive time and come out at the cylinder top, denoted $Out(\sigma_1)$. Solutions starting at interior points of $Out(\sigma_1)$ go outside the cylindrical neighbourhood in negative time. In these coordinates, the invariant manifold $W^{u}(\sigma_1)$ is the $z$--axis, intersecting $Out(\sigma_1)$ at the origin of the polar coordinates. Reversing the time, we get dual statements for $\sigma_2$. After linearization, the one-dimensional invariant manifold of $\sigma_2$, $W^{s}(\sigma_2)$, is the $z$--axis, intersecting the top of the cylinder, $In(\sigma_2)$, at the origin of its coordinates. Trajectories starting at  interior points of $In(\sigma_2)$
go into $V_2$. Trajectories starting at interior points of the cylinder wall  $Out(\sigma_2)$
go into $V_2$ in negative time. The set $Out(\sigma_2)\cap W^{u}(\Sigma_2)$ is parametrized by $y=0$.
Trajectories that start at 
$In(\sigma_2)\backslash W^s(\sigma_2 )$  leave the cylinder $V_2$ at $Out(\sigma_2)$.

\begin{figure}
\begin{center}
\includegraphics[height=8.2cm]{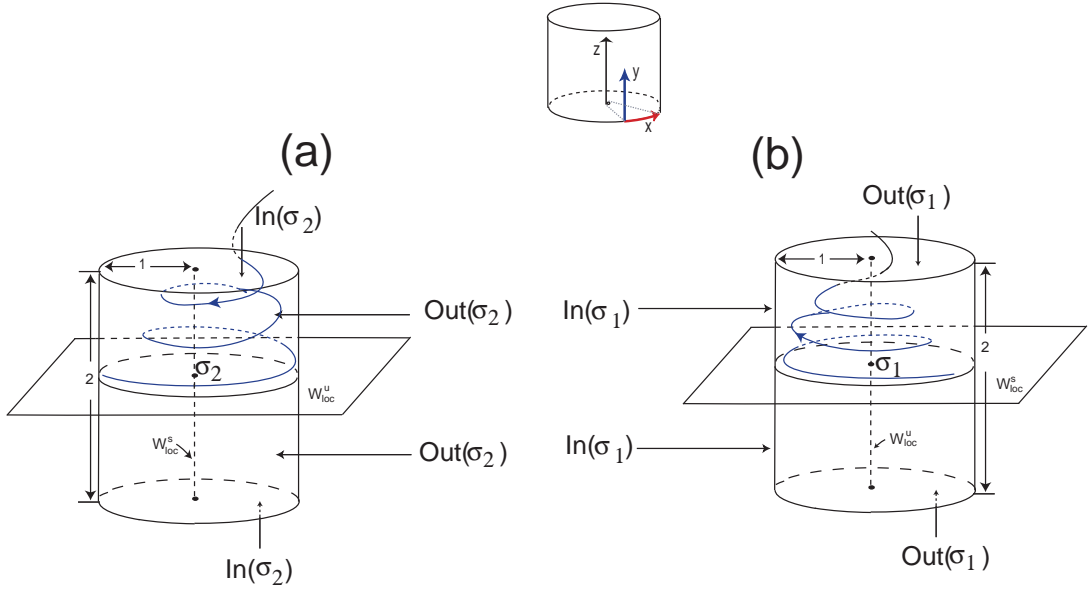}
\end{center}
\caption{\small Local cylindrical neighbourhoods $V_2$ and $V_1$ near $\sigma_2$ and $\sigma_1$, respectively. The relative confi\-guration of the saddle-foci is the same as in Figure \ref{conservativo2}.}
\label{conservativo3}
\end{figure}

\subsection{Conservative transition maps}\label{CTM}
Let $x\in M$ be a regular point, $T>0$ and $$\mathcal{P}_{F}^{T}(x):\mathcal{V}_{x}\subset\mathcal{N}_{x}\rightarrow\mathcal{N}_{\varphi(T,x)}$$ 
be the Poincar\'{e} map, where $\mathcal{N}_{\varphi(s,x)}= F(\varphi(s,x))^\perp$, $s=0,T$, is the surface contained in $M$ whose tangent space at $\varphi(s,x)$ is generated by $F(\varphi(s,x))^\perp$, $s=0,T$, and $\mathcal{V}_{x}$ is a small neighborhood in $\mathcal{N}_{x}$ of $x$. We can guarantee the existence of a continuous time-$t$ arrival function $\tau(x,t)(\cdot)$ from $\mathcal{V}_{x}$ into $\mathcal{N}_{\varphi(T,x)}$. If $\mathcal{V}_{x}\subset\mathcal{N}_{x}$ is sufficiently small then: $$\mathscr{F}:=\bigcup_{0\leq t\leq\tau(x,t)(y),\,\,y\in\mathcal{V}_x}\mathcal{P}_{F}^{t}(x)(y)$$  is a thin flowbox with base $\mathcal{V}_x$ and top $\mathcal{P}_{F}^{\tau(x,t)(y)}(x)(y)$. Due to the conservative flow-box Theorem (see \cite[Lemma 3.4]{Bessa_ETDS}), with a volume-preserving change of coordinates, we can trivialize $\mathscr{F}$ into a cylinder and thus we may obtain an expression for the global map $\Psi_{1,2}: Out (\sigma_1)  \rightarrow In (\sigma_2)$. Next lemma is the conservative version of a result stated in Bykov \cite{Bykov} and its proof directly follows by applying \cite[\S3]{Bessa_ETDS} and \cite[\S2.2]{BessaRocha}: 

\begin{lemma} \label{Perturb}
There are a volume-preserving change of coordinates in $Out (\sigma_1)$ such that the linear part of the global map $\Psi_{1,2}$ has the form:
$$\Psi_{1,2}(x,y)=\begin{pmatrix} a & 0\\ 0 & \frac{1}{a}\end{pmatrix} \begin{pmatrix} x\\ y\end{pmatrix} $$ for some $a\geq 1$. 
\end{lemma}

Consider the transition maps:
$$
\Psi_{1,2}:Out(\sigma_1) \longrightarrow In(\sigma_2)
\qquad 
\text{and} \quad \Psi_{2,1}:Out(\sigma_2) \longrightarrow In(\sigma_1).
$$

By Lemma ~\ref{Perturb}, the first order (linear) approximation of $\Psi_{1, 2}$, up to a change of coordinates, it is equal to:

\begin{equation}
 \Psi_{1,2}(x,y)=\left(ax, \frac{1}{a}y\right)
\qquad a \geq 1.
 \label{transition}
 \end{equation}

The transition $\Psi_{2,1}$ can be seen as a rotation by an angle $\alpha(\lambda)$.
As in \cite{LR3, Rodrigues2}, we assume that $\alpha \equiv \frac{\pi}{2}$, simplifying our computations.

\subsection{Local maps near the saddles}

The flow is transverse to the previous cross sections and the boundaries of $V_1$ and of $V_2$ may be written as the closures of the disjoint unions $In(\sigma_1) \cup Out (\sigma_1)$ and  $In(\sigma_2) \cup Out (\sigma_2)$, respectively. The trajectory of  the point $(x,y)$ in $In(\sigma_1) \backslash W^s(\sigma_1)$ leaves $V_1$ at $Out(\sigma_1)$ at:
$$
\begin{array}{c}
\Phi_{1}(x,y)=\left(\sqrt{|y|},-g_1 \ln y+x\right)=(r,\varphi),
\end{array}
$$
where $g_1=\frac{\alpha_{1 }}{E_{1}}>0$. In a dual way, points $(r,\varphi)$ in $In(\sigma_2) \backslash W^s(\sigma_2)$ leave $V_2$ at $Out(\sigma_2)$ at:
\begin{equation}
\begin{array}{c}
\Phi_{2}(r,\varphi )=\left(-g_2\ln
r+\varphi,r^{2}\right)=(x,y),
\end{array}
\end{equation}
where $g_2=-\frac{\alpha_{2 }}{E_{2 }}<0$. 
\begin{remark}
The  minus sign in the equation $\dot{\theta}= - \alpha_{2}$ of  \eqref{local2} suggests the condition stated on hypothesis \textbf{(H4)} about different chirality.
\end{remark}

\subsection{Global return map}
\label{Geom_vw}
Now we describe the geometry associated to the local dynamics near each equilibrium. We start with some useful definitions adapted from \cite{ALR, LR3}. 
A continuous real valued map defined on an interval $[0,1]$ is \emph{quasi-monotonic} if it is monotonic for a subinterval $[0,a]\subseteq [0,1]$ where $a\leq 1$.

\begin{definition} We shall define three objects to be used in the sequel: segments, spirals and helices:
\begin{enumerate}
\item \label{def_segmento} A \emph{ segment} $\beta $ 
 in  $In(\sigma_1)$ or $Out(\sigma_2)$ is a smooth regular parametrized curve of the type 
$$\beta :[0,1]\rightarrow In(\sigma_1) \qquad \text{or} \qquad \beta :(0,1]\rightarrow Out(\sigma_2)$$ that meets $W^{s}_{loc}(\sigma_1 )$ or $W^{u}_{loc}(\sigma_2 )$ transversely at the point $\beta (0)$ only and such that, writing $$\beta (s)=(x(s),y(s)),$$
both $x$ and $y$ are quasi-monotonic and bounded functions of $s$ and $\frac{dx}{ds}$ is bounded.
\bigbreak

\item \label{def_spiral}
A \emph{spiral} in $Out(\sigma_1)$ (resp.: $In(\sigma_2)$)  around  a point $p\in Out(\sigma_1)$(resp.: $p\in In(\sigma_2))$ is a curve 
$$\alpha :(0,1]\rightarrow Out(\sigma_1) \qquad \text{or} \qquad \alpha :(0,1]\rightarrow In(\sigma_2)$$ 
satisfying $\lim_{s\to 0^+}\alpha (s)=p$ and such that, if 
$\alpha (s)=(\alpha _{1}(s),\alpha _{2}(s))$ are its expressions  in
polar coordinates $(\rho ,\theta )$ around $p$, then 
$\alpha _{1}$ and $\alpha _{2}$ are quasi-monotonic, with $\lim_{s\to 0^+}|\alpha _{2}(s)|=+\infty$, i.e., winds infinitely many times around the point $p$.
\bigbreak
\item  Consider a cylinder $C$ parametrized by a covering $(\theta,h )\in  \mathbb{R}\times[a,b]$,
with $a<b\in\mathbb{R}$ where $\theta $ is periodic.
A \emph{helix} in the cylinder $C$ 
\emph{accumulating on the circle} 
$h=h_{0}$ is a curve
$\gamma :[0,1]\rightarrow C$
such that its coordinates $(\theta (s),h(s))$ 
satisfy $$ \lim_{s\to 0^+}h(s)=h_{0} \qquad \text{,} \qquad\lim_{s\to 0^+}|\theta (s)|=+\infty$$ and the map $h$ is quasi-monotonic.
\end{enumerate}
\end{definition}

 The next result summarizes some basic ideas on the spiralling geometry of the flow near the saddle-foci and has been proved in \cite[Section 6]{ALR}.  Since in its original form the authors assumed implicitly that the Property \textbf{(H4)} does not hold, we make the necessary reformulations.

\begin{lemma}
\label{Structures}
A segment $\beta $:
\begin{enumerate}
\item  
in $In(\sigma_1)$ is mapped by $\Phi _{1}$ into  a spiral in $Out(\sigma_1)$ around $W^u(\sigma_1)\cap Out(\sigma_1)$;
\item 
in $Out(\sigma_2)$ is mapped by $\Phi _{2}^{-1}$ into a spiral in $In(\sigma_2)$ around $W^s(\sigma_2)\cap In(\sigma_2)$;
\item \label{item3}
in $In(\sigma_1)$ is mapped by $\Phi _{1}$ into  a spiral in $Out(\sigma_1)$ around $W^u(\sigma_1)\cap Out(\sigma_1)$,
that is mapped by the conservative transition $ \Psi_{1,2}$  into another spiral, around $W^s(\sigma_2)\cap In(\sigma_1)$. This new spiral suffers the effect of the hyperbolic saddle distortion given by Lemma~\ref{Perturb}. 
\item If \textbf{(H4)} does not hold, the spiral defined on (\ref{item3}) is mapped by $\Phi _{2}$ into a helix in $Out(\sigma_2)$ accumulating uniformly on the circle  $Out(\sigma_2) \cap W^{u}(\sigma_2)$.
 \end{enumerate}
\end{lemma}

The complete study where hypothesis \textbf{(H4)} does not hold has been done in \cite{LR} for dissipative systems.

\section{Proof of Theorem \ref{Main1}: the reversals property and primary tangencies}
\label{SectionT-pointHetero}

We assume that $W^u(\sigma_2) \cap In(\sigma_1)$ and $W^s(\sigma_1) \cap Out(\sigma_2)$ define vertical lines across the cylinder walls of $V_1$ and $V_2$, $In(\sigma_1)$ and $Out(\sigma_2)$, respectively. 
Let $\beta(s)=(x_1(s), y_1(s))=(0,s) \subset In(\sigma_1)$, $s \in [0,1]$, be a parametrization of $W^u(\sigma_2) \cap In(\sigma_1)$, where $(0,0)$ are the coordinates of the point $[\sigma_2 \rightarrow \sigma_1] \cap In(\sigma_1)$. Let $(x_2(s), y_2(s))$ be the coordinates of $\Phi_{2 }\circ \Psi_{1,2 }\circ \Phi_{1 }(\beta(s))\in Out(\sigma_2)$. 
 The set $\Phi_{2} \circ \Psi_{1,2} \circ \Phi_{1}(\beta(s))$ ($s \neq 0$) defines an oriented curve in $Out(\sigma_2)$. Under hypothesis \textbf{(H4)}, it changes the direction of its turning around the cylinder $Out(\sigma_2)$ at infinitely many points where it has a vertical tangent. In what follows, we obtain the expression for $\Phi_{2} \circ \Psi_{1,2} \circ \Phi_{1 }(\beta(s))$, $s \neq 0$. The proof runs along the same lines as in \cite{LR3}.

\begin{lemma}
\label{x,y}
Let $\beta(s)=(0, s)$, $s \in [0,1]$, be a segment in $In(\sigma_1)$ parametrized by $s$. Then, for $s \neq 0$:

\begin{equation}\label{xweyw}
\left\{
\begin{array}{l}
x_2(s)= 
-\frac{1}{2} g_2 (\ln s [C(\varphi)])+ \Phi(\varphi)\\
 \\
y_2(s) = 
 s [C(\varphi)]\\
\end{array}
\right.
\end{equation}
where
\begin{equation}\label{CePhi}
\varphi(s)=-g_1 \ln s, \quad C(\varphi)= a^2 \cos^2(\varphi) + \frac{1}{a^2}  \sin^2(\varphi) \quad
\end{equation}
and 
$$
\\ \quad \Phi(\varphi)=\arg \left( c_1 s^{\delta_1}\left(a\cos(\varphi),\frac{1}{a} \sin(\varphi)\right) \right)=
 \arg \left( a\cos(\varphi),\frac{1}{a} \sin(\varphi)\right),
$$
with the argument $\arg$  taken in the  interval $\left[\frac{k}{2}, \frac{(k+1)}{2}\right]$, $k \in \mathbb{Z}$, that contains $\varphi$.
\end{lemma}

From now on, denote by $\eta$ the following map $$\Phi_{2 }\circ \Psi_{1,2 }\circ \Phi_{1 }: In(\sigma_1) \rightarrow Out(\sigma_2).$$

\begin{figure}
\begin{center}
\includegraphics[height=7cm]{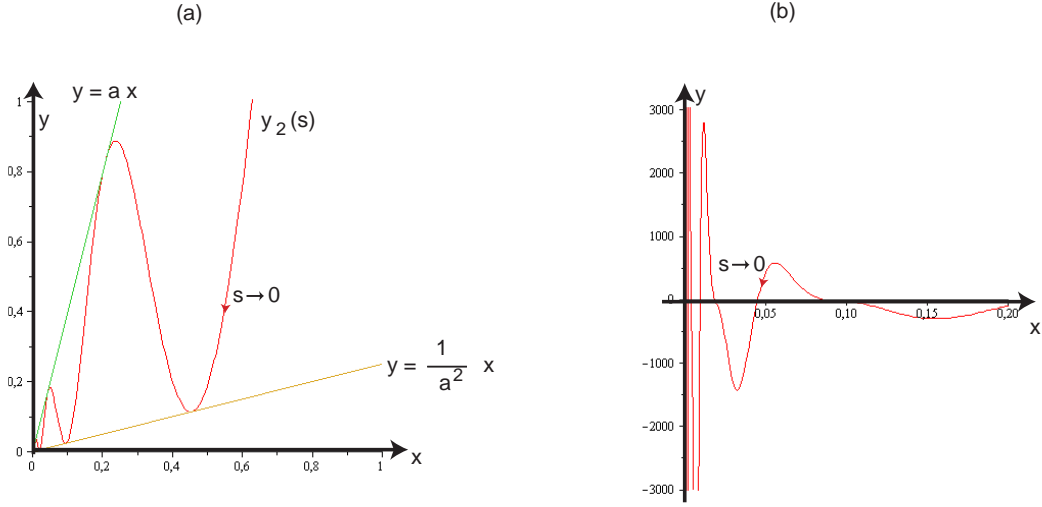}
\end{center}
\caption{\small Graphs of (a): the coordinate $y_2(s)$ (height) and (b); the coordinate $x_2(s)$ (angular coordinate), with $a \neq 1$, using Maple. Parameters: $\alpha_1=\alpha_2=1$, $a=E_1=2$ and $C_1=E_2=1$.}
\label{maple1}
\end{figure}

\subsection{Degenerate case: $a=1$}
If $a=1$, then $C(\varphi)=1$ and $y_2(s)=s$. Moreover,
$$
x_2(s)= 
- \frac{1}{2}g_2 (\ln s)+ \varphi=- \frac{1}{2} g_2 (\ln s) -g_1 \ln (s)=\left(-\frac{1}{2}g_2-g_1\right)\ln(s)
$$

If $g_2=-2g_1$ as in the reversible case reported in \cite{KLW, LTW}, then $x_2$ is constant and the segment $\beta$ is mapped under $\eta$ into a vertical segment.

\subsection{General case: $a \neq 1$}
As suggested in Figure \ref{maple1}(a), if $a \neq 1$, then the coordinate $y_2$ is not a monotonic function of $s$ but $\lim_{s \rightarrow 0^+} y_2(s)=0$. Moreover, for any $a \neq 0$, we have:
$$\text{if} \qquad  -\frac{1}{2}g_2-g_1> 0 \qquad \text{then}\quad \lim_{s \rightarrow 0^+} x_2(s)=+\infty $$ {and} \qquad $$\text{if} \qquad -\frac{1}{2}g_2-g_1<0 \qquad \text{then} \qquad \lim_{s \rightarrow 0^+} x_2(s)=-\infty.$$

\bigbreak

Under hypotheses \textbf{\textbf{(H1)}}--\textbf{(H5)}, we show that the coordinate map 
$x_2$ is not a monotonic  function of $s$, since the curve $\eta \circ \beta$ reverses the direction of its turning around $Out(\sigma_2)$ infinitely many times -- see Figure \ref{maple1}(b). This is the notion suggested by the following definition.

\begin{definition}
We say that the vector field $F$ has the \emph{dense reversals property} if for the vertical segment $$\beta(s)=(0, s)\in In(\sigma_1),\qquad s \in [0, 1],$$ the projection into $W^u_{loc}(\sigma_2)$ of the points where $\eta\circ \beta$ has a vertical tangent is dense in $W_{loc}^u(\sigma_2)\cap Out(\sigma_2)$.
\end{definition}

The dense reversals property is the key step in the proof of Theorem \ref{Main1}.
In order to prove it, we use the assumption \textbf{(H5)} on  the parameters $P=\left(\alpha_1, C_1, E_1, \alpha_2, C_2, E_2\right)$ 
that determine the linear part of the vector field $F$ at the equilibria.

 \begin{figure}\begin{center}
\includegraphics[height=5cm]{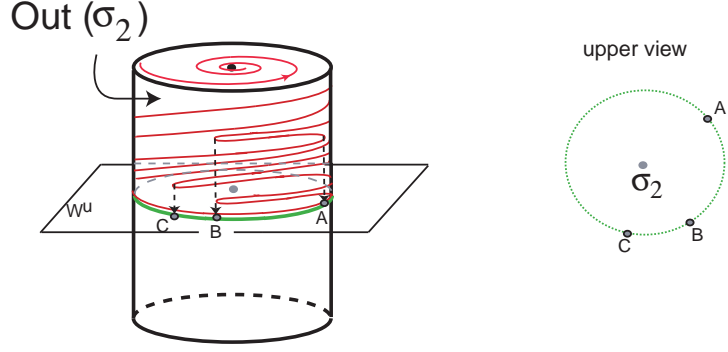}
\end{center}
\caption{\small On the cross section $Out(\sigma_2)$, we observe the line $W^u(\sigma_2)$ reversing the orientation of its angular coordinate. The set $W_{loc}^u(\sigma_2)\cap Out(\sigma_2)$ is diffeomorphic to a circle.}
\label{reversions}
\end{figure}

\begin{proof}[Proof of Theorem~\ref{Main1}] The proof will be divided into five parts, \textbf{(I)} -- \textbf{(V)}. The first part of the proof runs along the same lines to that of \cite{LR3}, which gives an optimal condition for the existence of tangencies. Let $F\in\mathfrak{X}_{\mu, P}^1(M)$.
\medbreak
\textbf{(I)}  
We need to compute the coordinate $x_2(s)$ at the points where $\eta(\beta(s))$ has a vertical tangent. Differentiating the expression \eqref{xweyw} of Lemma \ref{x,y} with respect to $s$, we get:
$$
\varphi'(s)=-\frac{g_1}{s},
\qquad 
\text{and}
\qquad
c'(\varphi)= \frac{2g_1}{s}\left(a^2-\frac{1}{a^2}\right)\cos(\varphi)\sin(\varphi)
$$
and then
\begin{equation}\label{eqdxw}
\frac{dx_2}{ds}=-\frac{1}{s}\left[\frac{g_2}{2} + \frac{1}{C(\varphi)}
 \left(2g_1 g_2 \left(a^2-\frac{1}{a^2}\right)\sin\varphi \cos\varphi + g_1\right)\right].
\end{equation}
Therefore $\frac{dx_2}{ds}=0$ has solutions if and only if 
$$
\frac{g_2}{2} + \frac{1}{C(\varphi)}
 \left(2g_1 g_2 \left(a^2-\frac{1}{a^2}\right)\sin\varphi \cos\varphi + g_1\right)=0.
$$
Define now the $\pi$--periodic map:
\begin{equation}\label{AdePhi}
A(\varphi)= 
E_1 a^2 \cos^2\varphi + \frac{E_1}{a^2}\sin^2\varphi +
 \alpha_1 \left(a^2-\frac{1}{a^2}\right)\sin\varphi\cos\varphi.
\end{equation}
 
\medbreak
\textbf{(II)} Now we are going to prove that the left part of (\ref{eqdxw}) has a zero if and only if $A(\varphi)=-\frac{\alpha_1 E_2}{\alpha_2}.$ 
Indeed, the expression
$$
\frac{g_2}{2} + \frac{1}{C(\varphi)}
 \left(2g_1 g_2 \left(a^2-\frac{1}{a^2}\right)\sin\varphi \cos\varphi + g_1\right)=0.
$$
is equivalent to
$$
g_2C(\varphi) + 2g_1g_2\left(a^2-\frac{1}{a^2}\right)\sin(\varphi)\cos(\varphi) =-2g_1.
$$
Using the definition of $g_1$ and $g_2$, the last equality yields the expression:
$$
\frac{\alpha_2}{E_2} a^2 \cos^2(\varphi)+2 \frac{\alpha_1\alpha_2}{E_1E_2}\left(a^2-\frac{1}{a^2}\right)\sin(\varphi)\cos(\varphi)+\frac{\alpha_2}{E_2 a^2}\sin^2(\varphi)=-2\frac{\alpha_1}{E_1},
$$
which is the same as
$$
E_1a^2\cos^2(\varphi)+2\alpha_1\left(a^2-\frac{1}{a^2}\right)\sin(\varphi)\cos(\varphi)+\frac{E_1}{a^2}\sin^2(\varphi)=-\frac{\alpha_1E_2}{\alpha_2}.
$$

Therefore the differential equation $\frac{dx_2}{ds}=0$ has solutions if and only if
\begin{equation}
\label{max}
\min A(\varphi) \leq -\frac{\alpha_1 E_2}{\alpha_2}\leq \max A(\varphi),
\end{equation}
corresponding to the condition that defines $\mathscr{D}$.  If $\varphi_0 \in [0, \pi]$ is a solution of  (\ref{eqdxw}), then there are infinitely many given by
$\varphi=\varphi_0+n\pi,$
where  $n \in \mathbb{Z}$. Since $\varphi=-g_1 \ln s$, then 
$\frac{dx_2}{ds}=0$ has solutions $$s_n=s_0e^{- \frac{n\pi}{g_1}}, \qquad n=0, 1,2, \ldots$$
where $s_0= e^{- \frac{\varphi_0}{g_1}}$. See Figure \ref{reversions}.
\medbreak

\textbf{(III)} For any $s_0\in\mathbb{R}$ and $n=0, 1,2, \ldots$, we have
\begin{equation}\label{xwsn}
x_2(s_n)=x_2\left(s_0 e^{\frac{-n\pi}{g_1}}\right)=x_2(s_0)+n\pi \left(1-\gamma\right)
\qquad\mbox{for}\qquad
\gamma=\frac{\alpha_2}{\alpha_1}\frac{C_1}{E_2} .
\end{equation}

 \begin{figure}\begin{center}
\includegraphics[height=4.0cm]{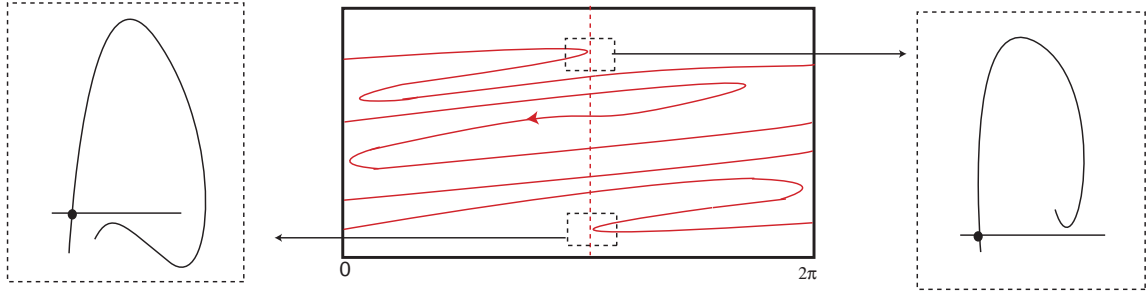}
\end{center}
\caption{\small On the cross section $Out(\sigma_2)$, we observe the $W^u(\sigma_2)$ reversing the orientation of its angular coordinate.  Two consecutive reversion points lie in region of the graph with different concavities, diving rise to different types of tangencies (and therefore to different types of dynamics). }
\label{Wall}
\end{figure}

Using the expressions \eqref{CePhi} we get that $$\varphi(s_n)=\varphi(s_0)+n\pi \qquad \text{and} \qquad
 C(\varphi+\pi)=C(\varphi).$$
Also, if $\varphi(s) \in \left[\frac{k}{2},\frac{(k+1)}{2}\right]$, then $\Phi(\varphi)$ lies in the same interval and 
$$\Phi(\varphi+1)=\Phi(\varphi)+1.$$
The result follows, using the expression \eqref{xweyw}  for $x_2(s)$ in Lemma \ref{x,y}. In Part \textbf{(III)}, we have shown that  $x_2\left(s_n\right)=x_2(s_0)+n\pi \left(1-\gamma\right).$
Hence, if the genericity condition $\gamma\notin\mathbb{Q}$ in \eqref{denseSet} holds, then the points $x_2(s_n)$ are dense in the circle defined by $W^u_{loc}(\sigma_2) \cap Out(\sigma_2)$ and thus $F$ has the dense reversals property.
With respect to the coordinate $x_2$, in the wall of the cylinder $V_2$, two consecutive reversion points lie in region of the graph with different concavities, as illustrated on Figure \ref{Wall}. This phenomenon implies the existence of chaotic dynamics near the tangency as we proceed to explain.

\medbreak

\textbf{(IV)}  If $x_0$ is  the projection of a point where  
$\eta\circ \beta$ has a vertical tangent, then  $W^s(\sigma_1)$ is tangent to $W^u(\sigma_2)$ at the corresponding point. Otherwise an arbitrarily small modification of the vector field $F$ in a neighbourhood of $(x_0,0)$ will move $x_0$, creating the tangency -- see Figure \ref{Cocoon}.  Suppose  $W^s(\sigma_1)\cap Out(\sigma_2)$ is a curve close to a vertical segment and that it is parametrized by  $$\xi(s)=(x(s),y(s)),$$with $\xi(0)=(x_0,0).$ 
Then,  changing  $R$ near $(x_0,0)$ if necessary, we may assume that $\xi(s)$ meets
$\eta \circ \beta$ at a point where the last curve  has a vertical tangent.
 The perturbation is achieved in part \textbf{(V)}.
\medbreak

\textbf{(V)} In ~\cite{LR3, LR2015} is it used an unfolding argument to obtain tangencies from the dense reversals property. Since we have no guarantee that a generic  unfolding process holds in the volume-preserving setting, we should tackle the problem using another approach. The key step is the volume-preserving Hayashi's Connecting Lemma \cite[\S1 and  Theorem E]{WX}, whose formulation may be adapted for flows as follows:

\begin{theorem}\cite[adapted]{WX}
\label{WeX}
Let $z\in M$ be a non-periodic solution of $F\in \mathfrak{X}_\mu^1(M)$. For any $C^1$-neighbourhood $\mathcal{U}\subset \mathfrak{X}_\mu^1(M)$ of $F$, there exists $\delta_0, L>0$ and $\rho>1$ such that for all $\delta\in(0,\delta_0]$ and any two points $p,q\in M$ outside the tube $\bigcup_{t\in[-L(\delta),0]}\varphi(t,(B(z,\delta))$, if the forward trajectory of $p$ intersects $B(z,\delta/\rho)$ and the backward trajectory of $q$ intersects $B(z,\delta/\rho)$, then there exists $Y\in\mathcal{U}$ such that $Y=F$ outside $\bigcup_{t\in[-L(\delta),0]}\varphi(t,(B(z,\delta))$ and there exists a one-dimensional connection $[p \rightarrow q]$.
\end{theorem}

The previous result says that if two distinct initial conditions $p$ and $q$ have solutions that visit a given neighbourhood of a point $z=\xi(0)$ and the initial conditions $p$ and $q$ are sufficiently far away from a piece of the backward solution of $z$, then we can find a divergence free vector field for which $p$ and $q$ are heteroclinically related. The support of the perturbation lies on $$\bigcup_{t\in[-L(\delta),0]}\varphi(t,(B(z,\delta)).$$
\medbreak
By \textbf{(I)--(III)},  we clearly have two points $p\in W^u(\sigma_2)$ and $q\in W^s(\sigma_1)$ in the conditions of Theorem \ref{WeX}. It is obvious that the solution associated of $z$ cannot intersect the heteroclinic connection $[\sigma_1\rightarrow \sigma_2]$. Therefore, we can always choose $\delta>0$ sufficiently small in order that the sets
$\bigcup_{t\in[-L,0]}\varphi(t,(B(z,\delta))$ and $[\sigma_1\rightarrow \sigma_2]$ are disjoint. Then, we can perform the volume-preserving perturbation given by Theorem~\ref{WeX} to obtain the tangencial intersection between $W^u(\sigma_2)$ and $W^s(\sigma_1)$ by making a connection near $z$. If the intersection is tangential the proof is completed. Otherwise, we consider a smooth  path $\{Z_s\}_{s\in[0,1]}$ of elements on $\mathfrak{X}_{\mu, P}^1(M)$ such that $Z_0=X$ and $Z_1=Y$. Since $Z_0$ displays no intersection between $W^u(Z_0,\sigma_2)$ and $W^s(Z_0,\sigma_1)$, $Z_1$ displays a transversal intersection between $W^u(Z_1,\sigma_2)$ and $W^s(Z_1,\sigma_1)$, and the accumulation is of a quadratic type, then there must be $s\in(0,1)$ such that $Z_s$ displays a tangencial intersection between $W^u(Z_s,\sigma_2)$ and $W^s(Z_s,\sigma_1)$.

\end{proof}

\section{Proof of Theorem \ref{Main2}: Generalized Cocooning cascade of heteroclinic tangencies}
\label{Cocoon_section}
The classic Cocoon bifurcation introduced in \cite{DIK, Lau} `begins' when the two-dimensional invariant manifolds of $\sigma_1$ and $\sigma_2$ have a tangency. After the first tangency, when moving slightly a given one-dimensional parameter, we obtain two structurally stable heteroclinic connections. It is the beginning of a sequence of heteroclinic bifurcations. Saddle-node bifurcations and elliptic solutions have been detected numerically in \cite{KE, Lau, Webster}.

\begin{proof}[Proof of Theorem \ref{Main2}]
We would like to prove that, for any $F\in \mathfrak{X}_{\mu,P}^1(M)$, and any tubular neighbourhood $N_\Sigma$ of the Bykov cycle $\Sigma$, for every $L>0$,  there exists a closed two-dimensional torus $\mathbb{T}^2$ such that $\sigma_{1,2} \notin \mathbb{T}^2$ and $F$ can be $C^1$-approximated by $F^\star \in \mathfrak{X}_{\mu,P}^1(M)$ whose flow displays a tangency between $W^u(\sigma_2)$ and $W^s(\sigma_1)$, intersecting twice $\mathbb{T}^2$ and having length greater than $L$ within $\mathbb{T}^2$.

Theorem~\ref{Main1} shows that the vector field of (\ref{general}) can be $C^1$-approximated by another vector field such that the associated flow displays a primary tangency inside $N_\Sigma$ between the manifolds $W^u(\sigma_2)$ and $W^s(\sigma_1)$. Beyond the tangency, there are infinitely many pieces of $W^u(\sigma_2)$ accumulating on $W^u(\sigma_2)\cap Out(\sigma_2)$. Therefore the set
 $ \Psi_{1,2}\circ \eta  (W^u(\sigma_2))$
  accumulates on $W^u(\sigma_2)\cap In(\sigma_1)$.  Repeating recursively the argument of the proof of Theorem \ref{Main1} in the appropriate direction,  we get curves in $Out(\sigma_2)$ with the dense reversals property. Using again the
      $C^1$-Connecting Lemma as in \textbf{(V)}, we may construct vector field $F^\star \in \mathfrak{X}_{\mu,P}^1(M)$ such that its flow exhibits tangencies of order $n$, with $n \in \NN$ large, and then we may define a torus $\mathbb{T}^2$, not containing the saddle-foci, inside which heteroclinic tangencies of order $n$ spend most of the time (say larger than $L>0$). An illustration of this phenomenon has been given in Figure \ref{Cocoon}.
 \end{proof}

 The main tool for detecting the manifolds $W^u(\sigma_2)$ and $W^s(\sigma_1)$ is the time delay function, defined by:
$$
T_{\pm}(x_0)= \left|  \int_{r_0}^r {dt}\right|, \qquad ||x(0)||=r_0
$$
where $T_\pm$ corresponds to integrate along the trajectory $$x(t)=x_0 + \int_{0}^{\pm t} f(x(s),c)ds \quad \text{with}  \quad x \in {M}.$$
The time delay function measures how long it takes for a given solution to escape from the sphere of radius $r>0$. Trajectories starting at the regions between two consecutive lines of Figure \ref{Cocoon} (corresponding to heteroclinic tangencies) will turn one more time around the equilibria $\sigma_2$ - more details in \cite[Section 7]{Rodrigues2015}. This is why the map $T_\pm$ acquires a castle shape with a higher level between pairs of equilibria  of the time delay map.

In the present paper, we cannot use the arguments of \cite{Rodrigues2} in which Lebesgue almost all solutions leave $N_\Sigma$ and hence will escape from a given sphere centered on the origin, with radius $r$. The later paper uses explicitly that hypothesis \textbf{(H4)} does not hold. 
By detecting logarithmic equilibria  of this map, one locates the invariant manifolds. A typical time delay function has several singularities as the one depicted in \cite[Fig. 11]{Lau}, corresponding to the heteroclinic tangencies. 
The time delay function acts as a numerical mechanism, searching along lines in a cross section, for pieces of the invariant manifolds of the two nodes.  Heteroclinic connections form a skeleton for the structure of periodic and bounded aperiodic orbits. According to \cite{Lau}, when elliptic periodic solutions are detected, the time delay function has flat tops.  Numerical results have been presented in \cite[Section 4]{DIKS}. 
 The analytical existence of these solutions is the core of next section.

\section{Proof of Theorem \ref{corollaryPeriodDoubling}: Horseshoes and Elliptic solutions}
\label{Horseshoes_section}

The organizing centre for the dynamics may be characterized by the presence of a hyperbolic set which is flow-invariant,
 indecomposable and topologically conjugate to a
Bernoulli shift space with  finitely many symbols. The union of all of these sets is not uniformly
hyperbolic, and countably many
of these sets are destroyed under small generic perturbations.  Homoclinic cycles of Shilnikov type and subsidiary cycles are expected to occur \cite{LR}.

\subsection{The existence of topological horseshoes}
In this section we study the existence of elliptic fixed points for the map $R$. We will make use of the references \cite{MGonchenko, GS}. The next result will be useful to make the bridge between conservative vector fields and conservative  first return maps.  We thanks Santiago Ib\'a\~nez for the hints on the next proof. 
\begin{lemma}\label{return_map}
\label{area-preserving}
{The return map of a volume-preserving vector field to any transverse section $S$ preserves the area $\delta dS$ (for some density $\delta$), and so is an area-preserving diffeomorphism.}
\end{lemma}

\begin{proof}
Let $F$ be a volume preserving vector field, \emph{i.e.} such that $\nabla \cdot  F = 0$. Take a transverse section $S$ with a well defined Poincar\'e map $P:S' \rightarrow S$, where $S' \subset S$. Now, take a simply connected subset $D \subset S'$ (assume it is a disk to make it simple). Consider the solid $M$ obtained by
saturation of the flow from $D$ to $P(D)$. The boundary of $M$, $\partial M$,  can be written as $\partial M = D \cup P(D) \cup \mathcal{T}$, 
where $\mathcal{T}$ is the saturation by the flow of the
one-dimensional boundary $B$ of $D \subset S'$ until it reaches $P(B)$.  By the Ostragadsky-Stokes Theorem \cite{Marsden, Webster}, the triple integral of $\nabla \cdot F$ on $M$ equals the surface integral of $f$ on $\partial M$. More formally:
$$
\iiint_M \nabla \cdot F= \iint_{\partial M} F. 
$$
Since $\nabla \cdot F=0$, the left part of the previous equality is zero. Moreover, $\int_\mathcal{T} F=0$ because $F$ is tangent to $\mathcal{T}$ at each point. Hence
$$
 \iint_D F + \iint_{P(D)} F=0.
$$

In the above expression, the normal vector on $D$ is the opposite to the normal vector on $P(D)$ because the normal vector must have a continuous variation on $\partial M$. By taking the normal vector pointing in the same direction the conclusion is that
$$
\iint_{P(D)} F=\iint_D F,
$$
meaning that there exists `an adapted' measure that is preserved (by $P$).
\end{proof}

\subsection{Suspended horseshoes and elliptic solutions}
\label{Suspended}
We will borrow the arguments of \cite{BDV,MGonchenko, GS} to show the existence of suspended horseshoes near the heteroclinic tangency. The latter authors compute the first return map $R$ and apply the Rescaling Lemma for studying bifurcations of fixed points.  The Rescaling Lemma  says that we may reduce the study of bifurcations of the first return maps to the classic analysis of bifurcations of the two-dimensional conservative H\'enon-like maps.

By \cite{GS}, we may find $M \in \mathbb{N}$ and a union of strips $\bigcup S_k \subset In(\sigma_1)$, $k>k_0 \in \NN$, accumulating on the segment $W^s(\sigma_1) \cap In(\sigma_1)$ as $k \rightarrow + \infty$ such that $R^m|_{\cup S_k}$ has a conservative horseshoe near the cycle, for $m>M$.   For small perturbations of $R$, infinitely many bifurcations of horseshoes creation or destruction occur, including the  birth and disappearance of elliptic periodic points. These horseshoe bifurcations must have different scenarios depending on a type of the initial tangency. The character of reciprocal position of the strips $S_k$ and their images $R(S_k)$ are essentially defined by the signs of two parameters that governs the linear part of the global map; the authors of \cite{GS} selected six different cases of  maps with quadratic homoclinic tangencies. For these parameters and for $k$ large, the position of all the strips $S_k$ and $R(S_k)$ have regular intersections. In particular,  if $F\in \mathfrak{X}_{\mu,P}^r(M)$, then the first return map associated to any vector field $C^1$-arbitrarily close to $F$ has a conservative horseshoe accumulating on $W^u(\sigma_2)\cap W^s(\sigma_1)$.  Moreover, there are persistent heteroclinic tangencies of the invariant manifolds of periodic solutions.

In a small section transverse to the cycle (not containing it), since the number of symbols of the Bernoulli shift approaches $\infty$ as the strip approaches the cycles, a cascade of bifurcations occurs, associated to the creation and annihilation of horseshoes \cite{YA}.  There are no guarantees that such perturbations still have the Bykov cycle in its flow.

\begin{remark}
If $A:\mathbb{R}^2 \rightarrow \mathbb{R}^2 $ is a linear map such that the $\det A=1$, then $\mathrm{Tr\ }A \in (-2,2)$ if and only if  its eigenvalues are complex (non-real) and conjugated. 
\end{remark}

\begin{proof}[Proof of Theorem \ref{corollaryPeriodDoubling}]
Let $F\in \mathfrak{X}_{\mu,P}^1(M)$.  Any neighbourhood of the points where we observe dense reversals property contains nontrivial hyperbolic subsets including infinitely many horseshoes, accumulating on the tangency. In particular, there are infinitely many 1-periodic solutions accumulating on the tangency.  Let $S\subset In(\sigma_1)$ and $$R=\Psi_{2,1}\circ\eta:S\rightarrow In(\sigma_1)$$ be the first return map of a vector field on $M$ satisfying (\textbf{H1})--(\textbf{H5}).
The points $(x(s),y(s))\in In (\sigma_1)$ for which we observe reversals satisfy
\begin{equation}
\label{tang_final}
A(\varphi)=\frac{\alpha_1 E_2}{\alpha_2}, \qquad \text{where} \qquad \varphi= -g_1\ln(s) \qquad \text{and} \qquad \det  DR(x,y)=1.
\end{equation}
 Observing that $1/a^2\le C(\varphi)\le a^2$ and $C(\varphi)$ is bounded, it follows that:
\begin{eqnarray*}
 \mathrm{Tr\ } DR(x,y)&=&
 2y\left(a^2-\frac{1}{a^2}\right) \sin\varphi\cos\varphi +
\frac{1}{y}\frac{\alpha_2}{E_1 E_2 C(\varphi)}\left( A(\varphi)-\frac{\alpha_1 E_2}{\alpha_2}
\right)=\\&=& o(y)+ \frac{1}{y}\frac{\alpha_2}{E_1 E_2 C(\varphi)}\left( A(\varphi)-\frac{\alpha_1 E_2}{\alpha_2}
\right)
\end{eqnarray*}

Near the initial conditions in $In(\sigma_1)$ where we observe the dense reversals property, condition (\ref{tang_final}) holds and thus, we conclude that $$\mathrm{Tr\ }DR(x,y)=o(y).$$ 

As illustrated in Figure \ref{trace}, for $y \approx 0$, 
if $\mathcal{A}$ is a neighbourhood in $Out(\sigma_2)$ of a point for which we observe a reversal, then there is a line where the trace is zero and that splits $\mathcal{A}$ into two connected components where the trace is positive or negative.
 In particular, we may define a strip $S\subset \mathcal{A}$ in which the trace belongs to the interval $(-2,2)$, where all 1-periodic solutions should be elliptic.  Intersecting the set of 1-periodic solutions that accumulate on the tangencies with the strip $S$, yields infinitely many elliptic 1-periodic solutions accumulating on the tangency.

 \begin{figure}\begin{center}
\includegraphics[height=3cm]{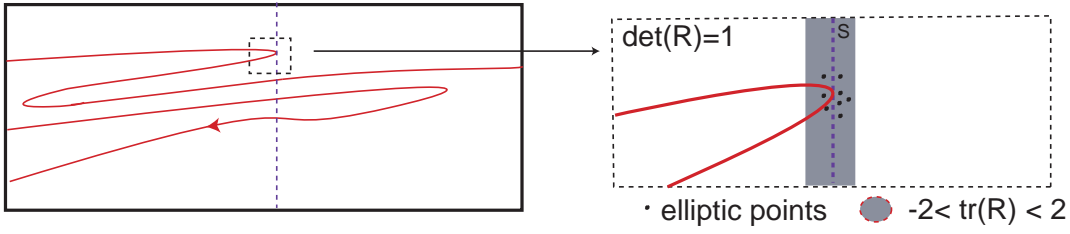}
\end{center}
\caption{\small Intersecting the set of 1-periodic solutions that accumulate on the tangencies with the strip $S$ yields infinitely many elliptic 1-periodic solutions accumulating on the tangency. }
\label{trace}
\end{figure}
 \end{proof}

\subsection{The Dynamics}
Dynamics near an elliptic point is reminiscent of the dynamics near a center type equilibrium: under regularity assumptions on the vector field, a generic elliptic point is surrounded by a family of closed KAM-curves which occupy a set of almost full Lebesgue measure in a small neighbourhood of the fixed point. On each KAM-curve the map $R$ is conjugate to an irrational rotation. As the distance to the fixed point increases, the rotation number changes in a monotonic way. The KAM-curves do not fill the neighbourhood of the elliptic point. Rational rotation numbers correspond to resonant Birkhoff zones between the closed curves and the dynamics in these regions is typically chaotic. This gives rises to chaotic dynamics between tori regions coexisting with hyperbolic horseshoes and the Bykov cycle. Adding the reversibility to the problem, symmetric elliptic trajectory may be seen as the limit of other elliptic periodic points -- see Gonchenko \emph{et al} \cite{GRios}. 

 When the fragile connection is broken, near the Bykov cycle there are a wide range of curious homoclinic bifurcations as reported in \cite{LR}.  Some of them do not depend on the chirality and are generated by the swirling pattern due to the existence of saddle-foci. They can be viewed as spiralling structures in a parameter diagram. Conservative structures might be comprised of elliptic islands separated by saddles. There is another self-similar organization of embedded saddles and centers on smaller scales. The closer one approaches a center of the rings, the greater the number of twists the outgoing heteroclinic connection makes before returning to the saddle -- see \cite{DIKS}.

\section{Discussion and concluding remarks}

In this paper, trying to answer the questions \textbf{(Q1)}--\textbf{(Q2)}, we have proved that the extended version of the Cocoon Bifurcations for conservative systems, the heteroclinic tangencies and the elliptic periodic solutions are strongly related. More precisely, in the present paper we have considered divergence-free vector fields whose flows have Bykov cycles, we have found heteroclinic tangencies leading to the description of elliptic periodic solutions for the first return map to a cross section defined near the tangency. Within a codimension 1 set in $\mathfrak{X}_{\mu}^1(M)$, tangencies between the invariant manifolds of the two saddle-foci densely occur. 

The proof of the existence of elliptic solutions near the horseshoes does not run along the same lines of Gonchenko and Shilnikov \cite{GS}, in which the authors used renormalisation to the H\'enon map. Along the proof of our results, we use the topological notion of chirality, making our findings   different from those of \cite{KLW, LR}. We asked for the $C^1$ volume-preserving Connecting Lemma \cite{WX} to perform local perturbations. Similar heteroclinic bifurcations have been considered in \cite{KLW, LTW}, where generic flat perturbations do not preserve the cycle.

The non-wandering dynamics near the Bykov cycle is dominated by hyperbolic horseshoes, conjugate to a full shift over a finite alphabet, that accumulate on the cycle. Hyperbolic behaviour cannot be separated from the ellitic behaviour by isotopies.

\subsection{The Conservative setting}
Our underlying hypothesis on the maintenance of the Bykov cycle is inspired in the $R$-reversible case, with $\dim Fix(R)=2$, where the mantainance holds \emph{per se}. Preserving the Bykov cycle is the most important consequence from the anti-symmetry; this is why we decided to reach general results without the reversibility hypothesis. The divergence-free character is crucial because it guarantees the stability of elliptic periodic solutions in our three-dimensional context. In higher dimensional case, the divergence freeness is no longer sufficient and we should
consider Hamiltonian vector fields  in order to get the stability of elliptic periodic solutions.

In the present paper, we are able to prove the existence of heteroclinic tangencies between $W^u(\sigma_2)$ and $W^s(\sigma_1)$. Our findings are richer that those of \cite{BessaRocha, BesDu07} in the sense that we can explicitly give the manifolds for which we observe tangency. Note that the works \cite{BessaRocha, BesDu07} proved the existence of a $C^1$-dense/residual set inside which tangencies and elliptic solutions occur via a  mechanism developed by Mora and Romero \cite{MR} to create open sets having a dense set of maps with tangencies nearby.

\subsection{The chirality}
Different chirality of the nodes is a new topological concept and has never been studied in the conservative setting.  This topological property has profound effects on the dynamics near the Bykov cycle.  Of course, in agreement with Kaloshin's Theorem on the prevalence of Kupka-Smale systems, these tangencies cannot occur for a full Lebesgue measure on $M$.  The density of tangencies near the cycle is markedly different from the results about the Michelson system proved in \cite{Chang, KE, McCord, Michelson} where the tangencies (near the Bykov cycle) are rare. If $\dim Fix(R)=1$, the $R$-reversibility prevents different chirality of the nodes.

The context of the proof of Theorem \ref{Main1} is different from that described in Knobloch \emph{et al} \cite{KLW} using spirals, a work that has been motivated by the Michelson system. Instead  of  restricting  the  flow  to the wall of the cylinder (as we do in Theorem \ref{Main1}), their conclusions, like Bykov's \cite{Bykov}, are supported by the study of spirals intersections. In \cite[Lemmas 6.6 and 6.7]{KLW} it is shown that two logarithmic spirals with a common centre can have at most two tangencies, assuming an irrational ratio of the real part of the complex eigenvalues -- see Figure \ref{spirals}(a). Although these spirals are only the leading order terms of the bifurcation equations, the density of tangencies cannot occur near the cycle, as shown in \cite{LR_proc, LR2015}. Without breaking the fragile connection, assuming the same chirality of the nodes,  tangencies may occur but not near the cycle.

Bykov implicitly assumed that the chiralities of two nodes are different in Formulas (3.4),(3.5) of \cite{Bykov}. Indeed, observe the same sign of the exponent of $e$ in the formulas that represent the evolution of the angular component of the spirals:
$$
\xi_1=d_1 e^{-\varphi_1/\omega_1}(1+\chi_{21}(0, \varphi_1,0))
\quad
\text{and}
\quad
\xi_2=d_2 e^{-\varphi_2/\omega_2}(1+\chi_{22}(0, \varphi_2,0)),
$$
where $\varphi_j$ represents the angular coordinates in $Out(\sigma_1)$ and $In(\sigma_2)$ and the nonreal eignevalues are given by $\alpha_j\pm i\omega_j$, $j\in\{1,2\}$. The same happens in Formula (3.1) of Bykov \cite{Bykov99}. After taking logarithms of both sides, from equality (3.3) on, Bykov neglects the assumption about the chiralities of the nodes, assuming that they are different. Therefore, in the cross sections, the spirals corresponding to the two-dimensional invariant manifolds of the saddle-foci are oriented in the same way. This explains the orientation of the spirals of Figure 2 of \cite{Bykov} in contrast to those depicted in Figure 11 of \cite{KLW}. These spirals are revived in Figure \ref{spirals}. Bykov never comments on the chiralities of the nodes, assuming implicitly that they are different.  

 \begin{figure}\begin{center}
\includegraphics[height=4.5cm]{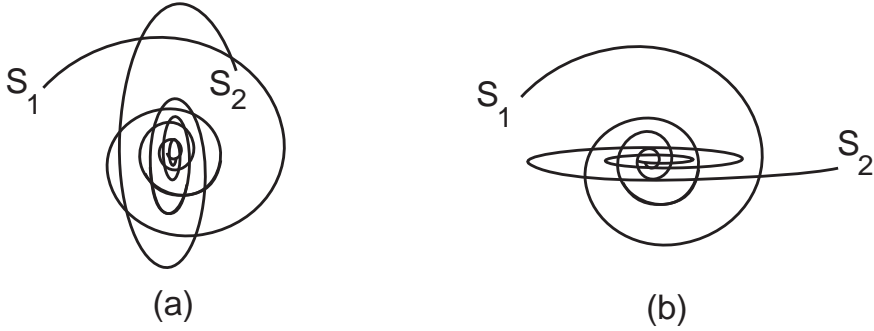}
\end{center}
\caption{\small The spirals $\mathcal{S}_1:=\Phi_2^{-1}(W^s(\sigma_1)\cap Out(\sigma_2))$ and $\mathcal{S}_2:=\Psi_{1,2}\circ \Phi_2(W^u(\sigma_2)\cap In(\sigma_1))$ have: Case (a): an infinite number of transversal intersections and at most two tangential intersections in total \cite{KLW}. Case (b): an infinite number of transversal intersections and, for small perturbations which do not break the fragile connection, an infinite tangential intersections.}
\label{spirals}
\end{figure}

\subsection{The Cocoon bifurcations}
The Cocoon Bifurcation is a set of rich bifurcation phenomena numerically observed by Lau \cite{Lau} in the Michelson system.
The authors in Kokubu \emph{et al} \cite{KWZ} proved the existence of the Cocoon Bifurcations for the Michelson system, which arise as the travelling wave equation of the Kuramoto-Sivashinky equation and also as a part of the limit family of the unfolding of the three-dimensional nilpotent singularity of codimension 3 -- see \cite{DIK1, DIKS}. Numerical results show that, when a given parameter varies, the system exhibits an infinite sequence of heteroclinic bifurcations due to the tangency of the two-dimensional invariant manifolds of saddle-foci. Each bifurcation creates new transverse heteroclinic connections. In \cite{DIK}, these bifurcations have been studied from a general point of view; they have been explained as a consequence of an organizing center called the \emph{cusp transverse heteroclinic chain}. Fixing the fragile connection, the Cocoon Bifurcations play the same role as the Newhouse bifurcations (generic unfoldings for the homoclinic tangency) in the dissipative category. Thick Cantor sets appear near the tangency, which are $C^r$-persistent, $r\geq2$.

In the present paper, we prove the existence of Cocoon bifurcations within a degenerate set of diver\-gence-free vector fields characterized by different chirality of the nodes. In contrast to the findings of \cite{DIK}, in which the authors could break the cycle, our perturbations have been performed without breaking the fragile connection. Furthermore, in \cite{DIK}, the reversibility is used to get saddle-nodes with both stable and unstable sets of dimension 1 and the volume-preserving character plays no role. We may use our results to illustrate the expected behaviour of the simple bifurcations of symmetric periodic solutions, namely saddle-node and period doubling via the creation and annihilation of horseshoes in a similar phenomenon to that described in \cite{BIR, LR2015, YA}. 

\subsection{Generic unfoldings}
Although we deal with a degenerate class of vector fields, our results hold for any generic unfolding of the cycle. The greatest difficulty arises when we want to preserve the Bykov cycle, a property that emerge in $R$-reversible systems with $\dim Fix(R)=2$. In a broader way, our results hold for any kind of cycles having a one-dimensional solution connecting two saddle-foci. In fact, following the approach in \cite{KLW}, it should be possible to show that vector fields containing Bykov cycles are always $C^1$-close to another with a tangency (and even a Cocconing cascade of tangencies), without making any assumption on the chirality. Nevertheless, there is no guarantee that the cycle is preserved. 

Although a degenerate class of systems satisfy \textbf{(H1)--(H5)}, the present work may be seen as the description of a simple and general mechanism to create non-hyperbolic behaviour in a robust fashion near a singular cycle.

\subsection{Final Remark}
Since the existence of cycles involving equilibria is not a generic phenomenon, they may exist outside the residual sets given by the classical dichotomies for divergence-free vector fields \cite{Bessa_ETDS, BessaRocha, BesDu07}. 
Several questions remain to be solved for this kind of cycles, in particular about the existence and sign of the Lyapunov exponents. The Lebesgue measure of the points that remain for all time in the neighbourhood of the cycle is an interesting point to be considered. In the spirit of \cite{LR, Rodrigues2}, conservativeness-breaking is a another direction that should be investigated. A first attempt to understand this phenomenon may be found in Dumortier \emph{et al} \cite[Section 4]{DIKS}. 

The present article should be seen as a starting point for further related studies; the ergodic characterization of the dynamics near these cycles (even in the dissipative case) is still far for being completely understood. 

\bigbreak
We are very grateful to the anonymous referee for the suggestions which have improved the article.



\begin{thebibliography}{99}


\bibitem{ALR}
Aguiar, M.A.D., Labouriau, I.S. \& Rodrigues, A.A.P. ``Switching near a heteroclinic network of rotating nodes,''
{\it Dyn. Syst} {\bf 25}(1), (2010)  75--95.





 \bibitem{ArMa}
Arbieto, A., Matheus, C. ``{A pasting lemma and some applications for conservative systems}''  {\it Ergod. Th. \&\ Dynam. Sys.} {\bf 27}, (2007) 1399--1417.

 

 \bibitem{Banyaga}
Banyaga, A.,  de la Llave, R. \& Wayne, C. E.  ``{A Cohomology Equations near Hyperbolic Points and Geometric Versions of Sternberg Linearization Theorem}''  {\it J. Geom. Anal.} {\bf 6}(4) (1996).

 

\bibitem{Bessa_ETDS}
Bessa, M.  ``{The Lyapunov exponents of generic zero
divergence three-dimensional vector fields}'', 
{\it Ergod. Th. \& Dynam. Sys.}, {\bf 27} (5), (2007) 1445--1472.




 
\bibitem{Bes} Bessa, M., ``A generic incompressible flow is topological mixing'', {\it C. R. Acad. Sci. Paris,} Ser. I 346, (2008) 1169--1174.

\bibitem{BessaRocha} Bessa, M. \& Rocha, J., ``{Homoclinic tangencies versus uniform hyperbolicity for conservative 3-flows}'', {\it Journal of Differential Equations},  {\bf 247},  (2009) 2913--2923.


\bibitem{BesDu07}
Bessa, M., Duarte, P., ``{Abundance of elliptic dynamics on conservative $3$-flows}'',
 {\it Dynamical Systems}, {\bf 23}, 4, (2008) 409--424.

 

\bibitem{BIR} Barrientos P., Ib\'a\~nez S. \&  Rodriguez, J. A.,``Robust heterodimensional cycles near conservative bifocal homoclinic orbits'', submitted  (2014).
 

\bibitem{BDV} Bonatti, C.,  D\'{i}az L. \& Viana, M., ``{Dynamics beyond uniform hyperbolicity. A global geometric and probabilistic perspective}'',  {\it Encycl. of Math. Sc},  \textbf{102}. Math. Phys. 3. Springer-Verlag  (2005).


\bibitem{Bykov93} Bykov, V.V., ``{The bifurcations of separatrix contours and chaos}'', {\it Physica D}, {\bf 62}, No.1-4, (1993) 290--299.

\bibitem{Bykov99} Bykov,V. V., ``{On systems with separatrix contour containing two saddle-foci}'', {\it J. Math. Sci.}, {\bf 95}, (1999) 2513--2522.


\bibitem{Bykov} Bykov, V., ``{Orbit Structure in a Neighbourhood of a Separatrix Cycle Containing Two Saddle-Foci}'', {\it Amer. Math. Soc. Transl}, {\bf  200}, (2000) 87--97.

\bibitem{Chang} Chang, H. C.,``{Travelling waves on fluid interfaces: normal form analysis of the Kuramoto-Sivashinsky equation}'', {\it Phys. Fluids}, {\bf 29},  (1986) 3142-- 3147.



\bibitem{Duarte} Duarte, P., ``{Abundance of elliptic isles at conservative bifurcations}, {\it Dynam. Stability Systems}, {\bf 14},  (1999) 339--356.


\bibitem{DIK1} Dumortier, F., Ib\'a\~nez, S. \& Kokubu, H., ``{ New aspects in the unfolding of the nilpotent equilibrium of codimension three}'', {\it Dynamical Systems}, {\bf 16},  (2001) 63--95.

\bibitem{DIK} Dumortier, F., Ib\'a\~nez, S. \& Kokubu, H., ``{Cocoon bifurcation in three-dimensional reversible vector fields}, {\it Nonlinearity}, {\bf 19},  (2006) 305--328.

\bibitem{DIKS} Dumortier, F., Ib\'a\~nez, S., Kokubu, H. \& Sim\'o, C.,  ``{About the unfolding of a Hopf-zero singularity}, {\it Disc. Cont. Dyn. Syst.}, {\bf 33}, 10, (2013) 4435--4471.


\bibitem{Field} Field, M.,``{Lectures on bifurcations, dynamics and symmetry}'', {\it Pitman Research Notes in Mathematics Series},  {\bf 356}, Longman (1996).

\bibitem{GS1} Gavrilov, N. \& Shilnikov, L. P., ``{On three-dimensional systems close to systems with a structurally unstable homoclinic curve: I}'', {\it Math.USSR-Sb.}, {\bf 17},  (1972) 467--485.

\bibitem{GS2} Gavrilov, N. \& Shilnikov, L. P., ``{On three-dimensional systems close to systems with a structurally unstable homoclinic curve: II}, {\it Math.USSR-Sb.}, {\bf 19},  (1973) 139--156.


\bibitem{GS} Glendinning, P. \& Sparrow, C., ``{T-points: a codimension two hetero clinic bifurcation}'', {\it J. Statist. Phys.}, {\bf 43}, (1986), 479--488.

\bibitem{MGonchenko} Gonchenko, M., ``{Homoclinic phenomena in conservative systems}'', Ph.D Thesis, Departament of Applied Maths, Universitat Politecnica de Catalunya  (2013).

\bibitem{Gonc83} Gonchenko, S. V., ``{On stable periodic motions in systems close to a system with a nontransversal homoclinic curve}'', {\it Russian Math. Notes}, {\bf 33},(5), (1983) 384--389.

\bibitem{GRios} Gonchenko, S. V., Lamb, J., Rios, I. \& Turaev, D. `` {Attractors and Repellers Near Generic Elliptic Points
of Reversible Maps},  Dokl. Akad. Nauk , {\bf 454}, no. 4, (2014),  375--378.

\bibitem{GS} Gonchenko, S. V. \& Shilnikov, L. P., ``{On two-dimensional area-preserving diffeomorphisms with infinitely many elliptic islands}'', {\it J. Stat. Phys.}, {\bf 101}, . (2000) 321--356.

\bibitem{GST99} Gonchenko, S. V.,  Shilnikov, L. P. \& Turaev, D. M., ''{ Homoclinic tangencies of an arbitrary order in Newhouse domains}'', {\it Itogi Nauki Tekh., Ser. Sovrem. Mat. Prilozh.}, {\bf 67}, (1999) 69--128 [English translation in J. Math. Sci., 105:1738-1778, 2001].

\bibitem{GP}
 Guillemin V. \&  Pollack A., ``{Differential Topology}'', {\it Prentice-Hall, Inc., Englewood Cliffs,} New Jersey,  (1974).


\bibitem{Katok}
 Katok, A. \&  Hasselblatt, B., ``Introduction to the Modern Theory of Dynamical Systems,''
 {\it Cambridge University Press}  (1995).




\bibitem{KE} Kent, P.  \& Elgin, J., ``{Travelling-wave solutions of the Kuramoto-Sivashinsky equation}'', {\it Nonlinearity}, {\bf 5},  (1982) 899--919.
 


\bibitem{KLW} Knobloch, J.,  Lamb, J. \& Webster, K., ``{Using Lin's method to solve Bykov's problems}'', {\it Journal of Differential Equations}, {\bf  257}, (2014) 2984--3047.

\bibitem{KWZ} Kokubu, H., Wilczak, D., \& Zgliczynsky, P., ``{Rigorous verification of coccon bifurcations in the Michelson system}, {\it Nonlinearity},. {\bf 20}, 9, (2008), 2147--2174.



\bibitem{LR} Labouriau, I. S. \& Rodrigues, A. A. P., ``{Global generic dynamics close to symmetry}'', {\it Journal of Differential Equations}, {\bf 253}, 8,  (2012) 2527--2557.

\bibitem{LR_proc} Labouriau,I. S.; \& Rodrigues, A. A. P., ``{Partial symmetry breaking and heteroclinic tangencies}'', in: S. Ib\'a\~nez, J.S. P\'erez del R\'io, A. Pumari\~no, J.A. Rodr\'iguez (Eds.), Progress and Challenges in Dynamical Systems, in: Springer Proc. Math. Stat., Springer-Verlag, (2013),  281--299.


\bibitem{LR3} Labouriau, I. S. \& Rodrigues, A.A.P., ``{Dense heteroclinic tangencies near a Bykov cycle}'', Journal of Differential Equations, {\bf 259}, (2015), 5875--5902.

\bibitem{LR2015} Labouriau, I. S. \& Rodrigues, A.A.P., ``{Global bifurcations close to symmetry}'', Submitted  (2015).


\bibitem{LS2004} Lamb, J. \& Stenkin, O., ``{Newhouse regions for reversible systems with infinitely many stable, unstable and elliptic periodic orbits}'', {\it Nonlinearity}, {\bf 17},  (2004) 1217--1244.

\bibitem{LTW} Lamb, J., Teixeira, M. \& Webster, K.,``{Heteroclinic bifurcations near Hopf-zero bifurcation in reversible vector fields in $\RR^3$}'', {\it J. Differential Equations}, {\bf 219}, (2005) 78--115.
 


\bibitem{Lau}  Lau, Y.-T., ``{The `cocoon' bifurcations in three-dimensional systems with two fixed points}'', {\it Internat. J. Bifur. Chaos Appl. Sci. Engrg.}, {\bf 2}, (1992) 543--558.

\bibitem{McCord} McCord, C. K.,``{Uniqueness of connecting orbits in the equation $Y^{(3)} = Y^2 ?$''}, {\it J. Math. Anal. Appl.}, {\bf 114}, (1986) 584--592.

\bibitem{Marsden} Marsden, J. \& Tromba, A.,``{Vector Calculus}''. 5th Edition W. H. Freeman  (2003).

\bibitem{Michelson} Michelson, D.,``{Steady solutions of the Kuramoto-Sivashinsky equation}'', {\it Phys. D}, (1986) 89--111.

\bibitem{MR} Mora, L., Romero, N., ``{Persistence of homoclinic tangencies for area-preserving maps}'', {\it Ann. Fac. Sci. Toulose Math.},{\bf 6(4)},  (1997) 711--725.

\bibitem{Mo} Moser, J.,``{On the volume elements on a manifold}'', {\it Trans. Amer. Math. Soc.}, {\bf 120}, (1965) 286--294.



\bibitem{Newhouse2} Newhouse, S. E., ``{Diffeomorphisms with infinitely many sinks}'', {\it Topology,} {\bf 13}, (1974) 9--18.

\bibitem{Newhouse3} Newhouse, S. E., ``{Quasi-elliptic periodic points in conservative dynamical systems}'', {\it Amer. J. of Math.}, {\bf 99},  (1977) 1061--1087.

\bibitem{Newhouse1} Newhouse, S. E., ``{The abundance of Wild Hyperbolic Sets and Non-Smooth Stable Sets for Diffeomorphisms}'', {\it Publ. Math. Inst. Hautes \'Etudes Sci.},  {\bf 50},  (1979) 101--151.




\bibitem{Rodrigues2} Rodrigues, A.A.P.,``{Repelling dynamics near a Bykov cycle}'', {\it J. Dynam. Differential Equations}, {\bf 25}, 3,  (2013) 605--625.


\bibitem{Rodrigues2015} Rodrigues, A.A.P.,``{Moduli for Heteroclinic Connections involving saddle-foci and periodic solutions}'', {\it Disc. Cont. Dyn. Syst.}, {\bf 35}, 7, 3155--3182, 2015


\bibitem{RL2014} Rodrigues, A. A. P. \& Labouriau, I. S.,``{Spiraling sets near a heteroclinic network}'', {\it Phys. D}, {\bf 268},  (2014) 34--49.




\bibitem{Webster} Webster, K., ``{Bifurcations In Reversible Systems With Application To The Michelson System}'', PhD. Thesis, Imperial College of London  (2005).


\bibitem{WX} 
Wen, L. \&  Xia, Z., ``{$C^1$ Connecting Lemmas}'', {\it Trans. Amer. Math. Soc. }, {\bf 352}, (2000) 5213--5230.


\bibitem{YA} Yorke, J. A. \& Alligood, K. T., ``{Cascades of period-doubling bifurcations: A prerequisite for horseshoes}'', {\it Bull. Am. Math. Soc. (N.S.)}, 9(3),  (1983) 319--322.

\end{thebibliography}
\end{document}